\documentclass[leqno,11pt,oneside]{amsart}
\usepackage[dvips]{color}
\usepackage[hypertex]{hyperref}
\usepackage{graphicx}
\usepackage{relsize}

\usepackage{enumitem}
\usepackage{amsmath,amsfonts,amssymb,amsthm}
\usepackage{bm}

\newcommand{\sing}{{\rm Sing}}

\usepackage[english]{babel}

\newtheorem{teo}{Theorem}[section]
\newtheorem{prop}[teo]{Proposition}

\newtheorem{example}[teo]{Example}
\newtheorem{cor}[teo]{Corollary}

\newtheorem{lem}[teo]{Lemma}
\newtheorem*{cor*}{Corollary}

\newtheorem{maintheorem}{Theorem}

\newtheorem*{lem*}{Lemma}

\newtheorem*{teorA'}{Theorem A'}

\theoremstyle{definition}

\newcommand{\R}{\mathbb{R}}
\newcommand{\Pe}{\mathbb{P}}

\newcommand{\Z}{\mathbb{Z}}
\newcommand{\Q}{\mathbb{Q}}

\newcommand{\C}{\mathbb{C}}
\newcommand{\CC}{\mathsmaller{\mathbb{C}}}

\newcommand{\A}{{\mathcal{A}}}

\newcommand{\D}{\mathcal{D}}
\newcommand{\F}{\mathcal{F}}

\newcommand{\G}{\mathcal{G}}
\newcommand{\LL}{{\mathcal L}}

\newcommand{\im}{{\rm Im}}
\newcommand{\re}{{\rm Re}}
\newcommand{\cl}[1]{\mbox{$\mathcal{#1}$}}

\newcommand{\codim}{{\rm codim}}

\newcommand{\HV}{{\rm Hol}_{\, vir}}
\newcommand{\diff}{{\rm Diff}}

\newcommand{\holv}{{\rm Hol}^{\, vir}}

\newcommand{\fun}[5]{\renewcommand{\arraystretch}{1.5}       \begin{array}{crcl}
       #1: & #2 &\rightarrow & #3 \\
       & #4 &\mapsto & #5
       \end{array} \renewcommand{\arraystretch}{1}
       }

\newcommand*\xbar[1]{ %
   \hbox{ %
     \vbox{%
       \hrule height 0.3pt 
       \kern0.35ex
       \hbox{%
         \kern-0.1em
         \ensuremath{#1}%
         \kern-0.1em
       }%
     }%
   }%
}

\newcommand*\xxbar[1]{%
   \hbox{%
     \vbox{%
       \hrule height 0.3pt 
       \kern0.4ex
       \hbox{%
         \kern-0.1em
         \ensuremath{#1}%
         \kern-0.1em
       }%
     }%
   }%
}

\begin{document}

\setcounter{section}{0}
\setcounter{teo}{0}
\setcounter{exe}{0}

\author{Arturo Fern\'andez-P\'erez \& Rog\'erio   Mol \& Rudy Rosas}
\title{On singular real analytic  Levi-flat foliations}
 \date{}

\thanks{First and second  authors partially financed   CNPq-Universal}
\maketitle

\begin{abstract}
A singular real analytic foliation $\mathcal{F}$ of real codimension one on an $n$-di\-men\-sional complex manifold $M$   is Levi-flat if each of its leaves is foliated by immersed complex manifolds of dimension $n-1$. These complex manifolds are leaves of a singular
real analytic foliation $\mathcal{L}$   which is tangent to $\mathcal{F}$.
In this article, we classify germs of Levi-flat foliations at $(\C^{n},0)$ under the hypothesis that $\mathcal{L}$ is a germ holomorphic foliation. Essentially, we prove that there are two possibilities for $\mathcal{L}$, from which the classification of $\F$ derives: either
it has a meromorphic first integral or is defined by a closed rational $1-$form.
Our local results also allow us to classify real algebraic Levi-flat foliations
on the complex projective space $\Pe^{n} = \Pe^{n}_{\C}$.
\end{abstract}


 \medskip \medskip

\section{Introduction}
Let $U \subset \C^{n}$ be an open subset and
 $H \subset U$ be a real analytic submanifold of real codimension one.
For each   $p \in H$, there is a unique complex vector space of dimension $n-1$ contained in the tangent space $T_{p}H$, which is  given by $\cl{L}_{p} = T_{p}H \cap i T_{p}H$, where $i = \sqrt{-1}$.
  When the real analytic distribution of complex hyperplanes $p \in H \mapsto \cl{L}_{p}$  is integrable, in the sense of Frobenius, we say that $H$ is a {\em Levi-flat} hypersurface.
Thus, $H$ is foliated by   immersed complex manifolds of dimension $n-1$, defining the so-called \emph{Levi foliation}.   The following  local normal form was proved by E. Cartan \cite[Th. IV]{cartan1933}: at each $p \in H$, there are   holomorphic coordinates $(z_{1},\ldots,z_{n})$ in a neighborhood $V$ of $p$
such that
\begin{equation}\label{formalocal-hlf}
H \cap V = \{\im(z_{n}) = 0\}.
\end{equation}
In particular, this says that the Levi foliation on $H \cap V$ is given by
$z_{n} = c$, for $c \in \R$.

Consider now   a (non-singular) real analytic   foliation  $\F$ of real codimension one    on the open subset $U \subset \C^{n}$.
We say that $\F$ is  \emph{Levi-flat} if its  leaves are, locally, Levi-flat hypersurfaces. Thus, there exists a real analytic   foliation $\LL = \LL(\F)$ on $U$ whose leaves are immersed complex manifolds of dimension $n-1$, entirely contained in the leaves of $\F$. Keeping the terminology of the hypersurface case, this underlying foliation is also  referred  to as \emph{Levi foliation}. The above definitions make sense both in the
local setting, as germs of Levi-flat hypersurfaces or foliations at $(\C^{n},0)$, or   in a global context,
as objects lying in an $n$-dimensional complex manifold $M$,  locally defined
in open coordinate  sets.

A basic example  is given by the  Levi-flat  foliation $\F_{0}$ defined   by level sets of the
 form $\re(z_{n}) = c$, for $c \in \R$,
in coordinates $(z_{1},\ldots,z_{n})$ of $\C^{n}$. Its associated
  Levi foliation $\LL_{0} = \LL(\F_{0})$ is given by $z_{n} = c$, for $c \in \C$. The germ of $\F_{0}$ at $(\C^{n},0)$   actually furnishes a local normal form for Levi-flat foliations under the context of CR-conjugations
\cite[Th. 4.1]{cerveau2004}. This means that, given a germ of Levi-flat foliation $\F$ at $(\C^{n},0)$
with  Levi foliation $\LL$,    there exists a germ of  real analytic diffeomorphism $\Phi: (\C^{n},0) \to (\C^{n},0)$  that, in some neighborhood of $0 \in \C^{n}$,  conjugates $\F$ and  $\F_{0}$
 together with  their Levi foliations $\LL$ and $\LL_{0}$, having the additional property  of being a holomorphic map
when restricted to each leaf of $\LL$ assuming values in the corresponding leaf of $\LL_{0}$.

In the singular case,    an irreducible real analytic   hypervariety $H \subset U \subset \C^{n}$ is said to be \emph{Levi-flat} if its \emph{regular part} --- the set of points of $H$ near which it is a real analytic manifold of real codimension one --- is a Levi-flat hypersurface.
We point out that, in this case, the Levi foliation of $H$ is not always the
restriction of a singular holomorphic foliation defined in a neighborhood of $H$ --- examples for this situation can be found in \cite{brunella2007} and \cite{fernandez2013}.
However, if the Levi foliation of a germ of real analytic Levi-flat hypervariety at $0 \in \C^{n}$ extends
to a holomorphic foliation in a neighborhood of $0 \in \C^{n}$, a theorem by Cerveau and Lins Neto \cite{cerveau2011} asserts that the ambient foliation has a meromorphic first integral,
that is, a non-constant meromorphic function that is constant along its leaves.

Likewise, we say that a  singular real analytic  foliation $\F$ is   \emph{Levi-flat} if it is
 Levi-flat
outside its singular set $\sing(\F)$.
Its \emph{Levi foliation} $\LL = \LL(\F)$ then  turns
out to be a singular real analytic foliation whose leaves are immersed complex manifolds  of  dimension $n-1$
(see Sect. \ref{section-preliminaries}).
Note that
 $\LL$ can be non-holomorphic --- an example can be found in \cite[Sect. 2]{cerveau2004}.
In this article, our purpose  is to classify Levi-flat foliations whose underlying Levi foliation is holomorphic --- we assume this fact throughout the text.
The  same hypothesis has  been considered in the article \cite{belko2003}.  Its results,
 in the local case, lead to the conclusion that   $\LL$ has a Liouvillean integrating factor, which   is   weaker then what our classification  proposes. We should also mention the paper \cite{landure2008}, were
Levi-flat foliations of class $C^{1}$ with holomorphic underlying foliation are studied.

As a first step,  we examine   Levi-flat foliations     defined by the levels of    germs  of
real analytic meromorphic functions  with real values. We obtain the following:
\begin{maintheorem}
\label{teo-tangent-to-levels}
A germ of holomorphic foliation of codimension one at $(\C^{n},0)$, $n \geq 2$,
 that is tangent to the levels of a non-constant real meromorphic function with real values
  admits a meromorphic first integral.
\end{maintheorem}
The proof of this Theorem \ref{teo-tangent-to-levels} is carried out in Sect. \ref{section-meromorphic-first}.
Let us describe it briefly.
First,
by  taking two-dimensional transversal sections, we can suppose that $n=2$. Denoting by $\G$ the holomorphic foliation in the theorem's statement,
if its   real meromorphic first integral has a fiber of codimension one
 accumulating to $0 \in \C^{2}$, we  have a real analytic
 Levi-flat hypersurface  tangent to $\G$ and the conclusion  comes   immediately from
  the aforementioned Cerveau-Lins Neto's theorem.
On the other hand, if such a fiber does not exist, we can conclude that $\G$ is a non-dicritical foliation
whose  leaves accumulating to $0 \in \C^{2}$ are   separatrices.    A dynamical study of the (virtual) holonomy of $\G$ then allows
us to construct   a holomorphic first integral.

The main result of this article --- for which Theorem \ref{teo-tangent-to-levels} is part of the proof --- provides
  a complete characterization  of local Levi-flat foliations with holomorphic Levi foliations:
\begin{maintheorem}
\label{teo-main-local}
Let $\F$ be a germ of real analytic Levi-flat foliation at $(\C^{n},0)$, $n \geq 2$, induced by
a real analytic $1-$form $\omega$.
If the Levi foliation $\LL$ is holomorphic, then at least one of the
following two possibilities occurs:
\begin{enumerate}[label=(\alph*)]
\item There exists a closed meromorphic $1-$form $\tau$
such that $ \omega  =  h |\psi|^{2} \re( \tau)$, where $\psi$ is a
holomorphic equation for the   polar set of $\tau$ and
$h$ is a germ of real analytic function with real values such that $h(0) \neq 0$;
 in this case, $\LL$ is induced by $\tau$.
\item There are a meromorphic function $\rho$ and a real meromorphic function $\kappa$, with $\kappa/ \bar{\kappa}$ constant
along the leaves of $\LL$,   such that
$ \omega  =   \re(\kappa d \rho)$; in this case, $\rho$ is a meromorphic first integral for $\LL$.
\end{enumerate}
\end{maintheorem}

We give an outline of the proof of Theorem \ref{teo-main-local}, which  will be detailed  in Sect. \ref{section-tangent-to-Levi-flat}.
Our main tools are the complexification   of analytic objects in $(\C^{n},0)$ to $(\C^{n} \times \C^{n*},0)$, where
$\C^{n*}$ is a copy of $\C^{n}$ with the opposite complex structure defined by the complex conjugation, and the associated mirroring or $(*)$-operator, which acts  on analytic objects in $(\C^{n} \times \C^{n*},0)$ (see
Sect. \ref{section-preliminaries}).
Since   the Levi foliation    $\LL$ is holomorphic, it turns out  that $\omega_{\CC}$, the complexification of $\omega$, belongs to a pencil of integrable holomorphic  $1-$forms (see Sect. \ref{section-pencil}). This geometric fact lies in the core of our proof, as long as we can   apply a local version of
Cerveau's classification of pencil of integrable $1-$forms \cite{cerveau2002}
(Proposition \ref{pencil-foliations}). An analysis   exploring   symmetries
under the $(*)$-operator  then  leads to the models proposed in
the assertion of our theorem.
Finally, in Sect. \ref{section-algebraic},   our local results are applied to the
classification of algebraic Levi-flat foliations on the complex projective space $\Pe^{n} = \Pe^{n}_{\C}$.
This is stated in Theorem \ref{teo-main-global}.

The authors are grateful to H. Reis for a comment that was essential in the development of
the arguments in the proof of Theorem \ref{teo-tangent-to-levels}.

\section{Foliations, Mirroring and complexification}
\label{section-preliminaries}

\subsection{Basic notation}

Consider   coordinates $z = (z_{1},\ldots,z_{n})$ in $\mathbb{C}^{n}$, where $z_{j} = x_{j} + i y_{j}$,
and the complex conjugation $\bar{z} = (\bar{z}_{1},\ldots,\bar{z}_{n})$, where $\bar{z}_{j} = x_{j} - i y_{j}$ and $i = \sqrt{-1}$.
We will employ the standard multi-index  notation.
For instance,   if $\mu = (\mu_{1},\ldots,\mu_{n}) \in \Z_{\geq 0}^{n}$, then $z^{\mu} = z_{1}^{\mu_{1}} \cdots z_{n}^{\mu_{n}}$
and $\bar{z}^{\mu} = \bar{z}_{1}^{\mu_{1}} \cdots \bar{z}_{n}^{\mu_{n}}$. Also,
if $I=\{i_1 < \ldots < i_p\} \subset \{1,\ldots,n\}$, then $|I|=p$,
$dz_I=dz_{i_1}\wedge \cdots \wedge dz_{i_p}$ and $d\bar{z}_I=d\bar{z}_{i_1}\wedge \cdots \wedge d\bar{z}_{i_p}$.

 We  fix the following notation
for rings of germs at   $(\C^{n},0)$:
\begin{itemize}
\item $\mathcal{O}_{n} =  \C\{z_{1},\ldots,z_{n}\}$ is the ring     of \emph{holomorphic functions};
\item $\A_{n} = \C\{z_{1},\ldots,z_{n},\bar{z}_{1},\ldots,\bar{z}_{n}\} = \C\{x_{1},y_{1},\ldots,x_{n},y_{n}\}$ is the ring   of \emph{real analytic functions} (with complex values);
\item $\A_{n \R} \subset \A_{n}$ is the ring of   \emph{real analytic functions with real values}.
\end{itemize}
To these rings, in the above order, we associate   the    fields of fractions:
\begin{itemize}
\item $\mathcal{M}_{n}$  is the field of \emph{meromorphic functions};
\item $\cl{Q}_{n}$ is the field of   \emph{real meromorphic functions} (with complex values);
\item $\cl{Q}_{n \R} \subset \cl{Q}_{n}$ is the field of   \emph{real meromorphic functions with real values}.
\end{itemize}
Note that $\phi \in \A_{n}$   is in $\A_{n \R}$   if and only if $\phi(z) = \xbar{\phi(z)}$ for every sufficiently small $z \in \C$. In terms of the Taylor series $\phi(z)=\sum_{\mu,\nu} a_{\mu\nu} z^{\mu}\bar{z}^{\nu}$, this  is equivalent to $a_{\mu\nu} =
\bar{a}_{\nu \mu}$ for all $\mu, \nu$. We make the following convention: dimensions or codimensions will be real, when  referring to   real objects (foliations, varieties), but will be complex, when the objects in question are complex.

\subsection{Levi-flat foliations}
\label{section-levi}

A germ of singular real analytic foliation $\F$ of real codimension one at $(\C^{n},0)$ is the
object induced by a $1-$form $\omega$ with real analytic coefficients without non-trivial
common factor in  $\A_{n \R}$, which satisfies the   Frobenius integrability condition  $\omega \wedge d \omega = 0$.
Writing
\begin{equation}
\label{omega-def}
\omega = \sum_{j=1}^{n} A_{j} d x_{j} + B_{j} d y_{j},
\end{equation}
where $A_{j}, B_{j} \in \A_{n \R}$ for $j = 1, \ldots, n$,
we have that the \emph{singular set} of $\F$,
\[ \sing(\F):= \sing(\omega) = \bigcap_{j=1}^{n}\{ A_{j}= B_{j} = 0 \}\]
 is a
real analytic set of real codimension at least two.

In complex coordinates $z = (z_{1},\cdots,z_{n})$  in $\C^{n}$, where
$x_{j} = \re(z_{j}) = (z_{j} + \bar{z}_{j})/2$ and
$y_{j} = \im(z_{j}) =  (z_{j} - \bar{z}_{j})/2 i$,
we have
\begin{equation}
\label{eq-omega}  
\omega     =    \sum_{j=1}^{n} \frac{A_{j} - i B_{j}}{2} dz_{j} +
\sum_{j=1}^{n} \frac{A_{j} + i B_{j}}{2} d \bar{z}_{j}
  =   \frac{\eta + \bar{\eta}}{2} =  \re(\eta
 ) .
\end{equation}
The $1-$form
\begin{equation}
\label{eta-def}
 \eta    =  \sum_{j=1}^{n} (A_{j} - i B_{j}) dz_{j}
\end{equation}
defines the (intrinsic)  distribution of complex hyperplanes (outside $\sing(\F)$) associated to $\omega$.
We also set
\begin{equation}
\label{eq-omega-sharp}
 \omega^{\sharp}    = \im (\eta )  = \frac{\eta - \bar{\eta}}{2 i}
  =    \sum_{j=1}^{n} -B_{j} dx_{j} + A_{j} dy_{j}.
 \end{equation}

As explained in the introduction, the foliation $\F$ is \emph{Levi-flat} if
 the distribution of complex hyperplanes
defined  outside $\sing(\F)$  by the $1-$form $\eta$ in \eqref{eta-def} is integrable, inducing the Levi foliation $\LL$   whose leaves are   immersed complex manifolds of codimension one.
 As a real foliation of codimension two, $\LL$ is defined by the Pfaff system
$ \omega = \omega^{\sharp} = 0$, which, by \eqref{eq-omega} and \eqref{eq-omega-sharp}, is equivalent  to the one defined by
$ \eta = \bar{\eta} =  0$ and, still, to the one induced by the real analytic $2-$form
$\eta \wedge \bar{\eta}$.
In this paper we deal with the case where $\LL$ is a holomorphic foliation.
This is can be characterized in the following way:
\begin{lem}
\label{lem-holomorphic-L}
Suppose that $\LL$ is a holomorphic foliation, defined by
an integrable holomorphic $1-$form $\sigma$    at $(\C^{n},0)$,
 with
  coefficients without common factors in  $\cl{O}_{n}$. Then
 there exists $\phi \in \A_{n}$ such that
$ \eta = \phi \sigma$.
\end{lem}
\begin{proof}
Write $\eta = \sum_{i=1}^{n} \varepsilon_{j} dz_{j}$,
where
 $\varepsilon_{j} \in \cl{A}_{n}$,
and
$\sigma = \sum_{i=1}^{n} \alpha_{j} dz_{j}$, where
  $\alpha_{j} \in \cl{O}_{n}$.
In a small neighborhood  of $0 \in \C^{n}$, for $z$ outside $\sing(\eta) \cup \sing(\sigma)$,
the equations
 $\eta(z) = 0$ and $\sigma(z) = 0$ define same hyperplane, so that
 there exists $\phi(z) \in \C^{*}$ such that $\eta(z) = \phi(z) \sigma(z)$. For each $j$ such that $\alpha_{j}(z) \neq 0$, we have that $\phi(z) = \varepsilon_{j}(z)/\alpha_{j}(z)$, so that $\phi$ extends to
 a neighborhood of $0 \in \C^{n}$ as a real meromorphic function that is real analytic outside
 $\sing(\sigma)$,  a complex analytic set of codimension two. This is possible, if and only if,
 $\alpha_{j}$ divides $\varepsilon_{j}$ in $\A_{n}$, whenever $\alpha_{j} \neq 0$. Consequently, $\phi$ is real analytic
 in a neighborhood of $0 \in \C^{n}$.
\end{proof}

\subsection{Mirroring}

Let $\mathbb{C}^{n*} \simeq \C^{n}$ be the  space obtained by endowing $\C^{n}$ with the   opposite complex structure, induced by the complex conjugation, where we   consider    complex coordinates $w = (w_{1},\ldots,w_{n}) = \bar{z}$, with $w_{j} = \bar{z}_{j} = x_{j} - iy_{j}$.   The conjugation map $z=x+iy \mapsto x-iy = w$ defines a biholomorphism between
 $\mathbb{C}^{n}$ and $\mathbb{C}^{n*}$ that we call \textit{mirroring}. We   define the \emph{$(*)$-operator},  taking sets, functions or differential forms   in $\C^{n}$ to
 their \emph{mirrors} in $\mathbb{C}^{n*}$,
 in the following way:
\begin{itemize}
\item If $\gamma \subset \C^{n}$ then
     $\gamma^{*} =   \{z ; \bar{z}\in \gamma\}\subset \mathbb{C}^{n*}.$
\item If $\phi: \gamma \subset \mathbb{C}^{n} \to \C$ is a function then
\[ \fun{\phi^{*}}{\gamma^{*} \subset \mathbb{C}^{n*}}{\C}{w}{\xxbar{\phi(\xbar{w})}
 .}\]
\item For a differential $p-$form $\varpi =\sum_{I,J}\phi_{I,J}(z)dz_I \wedge d\bar{z}_{J}$,
with $|I|+|J| = p$, we set
$$ \varpi^{*} =\sum_{I,J}\phi_{I,J}^{*}(w)dw_I \wedge d\bar{w}_{J}.$$
 \end{itemize}
Note that if   $\phi \in \A_{n}$   has Taylor series
$\phi(z)=\sum_{\mu,\nu}a_{\mu\nu}z^{\mu}\bar{z}^{\nu}$, then
\begin{equation}
\label{mirror-function}
\phi^{*}(w)= \xxbar{ \displaystyle \sum_{\mu,\nu}a_{\mu\nu}\bar{w}^{\mu}w^{\nu}} =  \sum_{\mu,\nu}\bar{a}_{\mu\nu}w^{\mu}\bar{w}^{\nu}.  
\end{equation}
Therefore, mirroring preserves the class of analyticity, real or complex.
As a consequence,   $\gamma \subset \mathbb{C}^{n}$ is  a real or complex analytic subvariety if and only if the same holds for $\gamma^{*} \subset \mathbb{C}^{n*}$.
Of our particular interest is the mirroring of foliations.
 If $\G$ is a holomorphic foliation of codimension  $p$  at $(\mathbb{C}^{n},0)$,
defined by a holomorphic $p-$form $\varpi$ --- that is integrable and locally decomposable
  outside the singular set --- then   $\G^{*}$ is the foliation
  defined by  $\varpi^*$. It is straightforward to check that the leaves of $\F^{*}$ are obtained
by   the mirroring of  those of $\mathcal{F}$.

The $(*)$-operator can be defined in a broader context,  acting in objects in $\C^{n} \times \C^{n*} \simeq \C^{2n}$. We have:
\begin{itemize}
\item If $\Gamma \subset \C^{n} \times \C^{n*}$, then $\Gamma^{*}  = \{(z,w); (\bar{w},\bar{z})  \in \Gamma \}$.
\item  $\Phi: \Gamma \subset \C^{n} \times \C^{n*} \to \C$ is a function, then
\[\fun{\Phi^{*}}{\Gamma^{*}}{\C}{(z,w)}{\xxbar{\Phi(\bar{w}, \bar{z})} \ .}\]
\item For a differential $p-$form
$\varpi=\sum_{I,I',J,J'}\phi_{I,I',J,J'}(z,w)dz_I \wedge d\bar{z}_{I'} \wedge dw_{J} \wedge d\bar{w}_{J'},$
where $|I| + |I'| + |J| + |J'|  = p$,
we define
$$\varpi^*=\sum_{I,I',J,J'}\phi^{*}_{I,I',J,J'}(z,w)dw_I \wedge d\bar{w}_{I'} \wedge dz_{J} \wedge d\bar{z}_{J'}.$$
\end{itemize}
Note in particular that, if
  $\Gamma = \gamma \times \{w\}$, with $\gamma \subset \C^{n}$ and $w \in \C^{n*}$, then $\Gamma^{*} = \{\bar{w}\} \times \gamma^{*}$.
For  a germ of holomorphic function $\Phi$  at $(\C^{n} \times \C^{n*},0)$ with Taylor series
$\Phi(z,w) = \sum_{\mu,\nu}a_{\mu\nu}z^{\mu} w^{\nu}$, we have
\begin{equation}
\label{eq-star-operator}
  \Phi^{*}(z,w) =
 \xxbar{\sum_{\mu,\nu}a_{\mu\nu}\bar{w}^{\mu} \bar{z}^{\nu} } =
  \sum_{\mu,\nu} \bar{a}_{\mu\nu} z^{\nu} w^{\mu}.
\end{equation}

\subsection{Complexification}

We work in $\mathbb{C}^{n} \times \mathbb{C}^{n*} \simeq \mathbb{C}^{2n}$ with
coordinates $(z,w)$. We assume henceforth that $\C^{2n}$ is endowed with this decomposition, in such a way that any function in $\cl{O}_{2n}$ or $\cl{M}_{2n}$ is intrinsically a function  in the coordinates $(z,w)$.
Let $\phi \in \A_{n}$ be a germ of real analytic function having Taylor   series
\begin{equation}
\label{eq-taylor-phi}
\phi(z,\bar{z})=\sum_{\mu,\nu} a_{\mu\nu} z^{\mu}\bar{z}^{\nu}.
\end{equation}
The \emph{complexification} of $\phi$ is the germ  of holomorphic function
$ \phi_{ \CC}  \in
\cl{O}_{2n}$
 whose Taylor series is
\begin{equation}
\label{complex-function}
\phi_{\CC}(z,w)=\sum_{\mu,\nu}a_{\mu\nu}z^{\mu} w^{\nu}.
\end{equation}
The complexification of a germ of  real analytic $p$-form  $\varpi=\sum_{I,J}\phi_{I,J}dz_I \wedge d\bar{z}_{J}$ at
$(\mathbb{C}^{n},0)$ is
the germ of holomorphic $p$-form at
$(\mathbb{C}^{2n},0)$ with expression
\begin{equation}
\label{complex-form}
\varpi_{\CC}=\sum_{I,J}\left(\phi_{I,J}\right)_{\CC}dz_I \wedge d w_{J}.
\end{equation}
This is evidently  a two-way process: the \emph{decomplexification} of a germ of holomorphic function $F \in \cl{O}_{2n}$ is the unique   $\phi \in \cl{A}_{n}$ such that $\phi_{\CC} = F$.
The \emph{decomplexification} of a holomorphic $p-$form $(\mathbb{C}^{2n},0)$ is done in an obvious way.
The ideas of complexification and decomplexification can be canonically extended
  to   meromorphic    functions and differential forms.

If $\phi \in \A_{n}$ is as in \eqref{eq-taylor-phi}, we have
\[\phi^{*}(w,\bar{w}) = \sum_{\mu,\nu}\bar{a}_{\mu\nu}w^{\mu} \bar{w}^{\nu},\]
which gives, by the correspondence $w = \bar{z}$,
\begin{equation}
\label{complex-function-star}
(\phi^{*})_{\CC}(z,w) =  \sum_{\mu,\nu}\bar{a}_{\mu\nu} z^{\nu} w^{\mu} .
\end{equation}
Comparing this with  \eqref{eq-star-operator} and   \eqref{complex-function}, we get the commutative property $\left(\phi^{*}\right)_{\CC} = \left(\phi_{\CC}\right)^{*}$, allowing us
to  denote both expressions by $\phi_{\CC}^{*}$. Similar remarks apply to successive   mirror and  complexification of
a real analytic $1-$form $\varpi$, making the notation $\varpi_{\CC}^{*}$   unambiguous.

We say that a function $F$ in $\cl{O}_{2n}$ is $(*)$-\emph{symmetric} or \emph{mirror symmetric} if
 $F^{*} = F$.
Observe that, for some  $\phi \in \cl{A}_{n}$,
the equality $\phi_{\CC} = \phi_{\CC}^{*}$ holds if and only if $\phi \in \A_{n \R}$.
Indeed, comparing \eqref{complex-function} and \eqref{complex-function-star}, we find
that the mirror invariance is equivalent to $a_{\mu\nu} = \bar{a}_{\nu\mu}$ for all indices $\mu, \nu$,
giving the conclusion. More generally, $(*)$-symmetric functions in $\cl{O}_{2n}$ are precisely those that decomplexify as   real analytic functions with real values.
 The following fact is straightforward: if $\varphi \in \cl{O}_{1}$ is a function
whose Taylor series has real coefficients and $F \in \cl{O}_{2n}$, then $(\varphi \circ F)^{*} = \varphi \circ F^{*}$.
In particular, $\varphi \circ F$ is $(*)$-symmetric if $F$ is. This fact can be applied,
 for instance, to the function  $\exp(F)$ and, when  $F$ is a unity such that $\re(F(0)) > 0$,  to $\sqrt{F}$ (taking the principal branch of the square root).
The notion of $(*)$-symmetry can be extended in an obvious way to meromorphic functions
at    $(\C^{2n},0)$.
 \begin{prop}
 \label{prop-meromorhpic-function}
   We have the following facts:
\par  (a)  A $(*)$-symmetric germ of meromorphic function in $\cl{M}_{2n}$
 can be written as a quotient of $(*)$-symmetric
holomorphic functions in $\cl{O}_{2n}$.
  \par  (b) A real meromorphic function $f \in \cl{Q}_{n}$ is in $\cl{Q}_{n \R}$ if and only if $f = \bar{f}$.
 \end{prop}
\begin{proof}
Let  $F$ be the meromorphic function as in item (a), written in
 the form $F = G/H$, where $G, H \in \cl{O}_{2n}$ are without common factors.
Since $F^{*} = F$, we have $G^{*}/H^{*} = G/H$. Comparing
  zeroes and poles, we find
a unity $U \in \cl{O}_{2n}$ satisfying $U U^{*} = 1$ such that $G^{*} = U G$
and $H^{*} = U H$. We have $|U(0)|=1$ and, in fact, we can suppose that $U(0) = 1$. Actually, taking $\alpha \in \C^{*}$ such that $\alpha^{2} = U(0)$, we have
\[ (\alpha G)^{*} = 
\frac{  U}{\alpha^{2}}(\alpha G) \qquad \text{and} \qquad  (\alpha H)^{*} = \frac{ U}{\alpha^{2}}(\alpha H).\]
Thus, it suffices to replace $G$ by $\alpha G$ and $H$ by $\alpha H$ in the rational decomposition of $F$.
Now, by putting $\tilde{G} = \sqrt{U} \, G$, we have
\[ \tilde{G}^{*}  = (\sqrt{U})^{*} G^{*} =  \sqrt{U^{*}} \, U G = \sqrt{U} \, G = \tilde{G} .\]
In a similar way, the function $\tilde{H} = \sqrt{U} \, H$ is such that $\tilde{H}^{*} = \tilde{H}$.
Therefore we can write $F = \tilde{G}/\tilde{H}$, where $\tilde{G}, \tilde{H} \in \cl{O}_{2n}$ are $(*)$-symmetric.

Item (b) follows straight by applying (a) to  the complexification $F = f_{\CC}$.

\end{proof}

We   state  the following result for future reference:
\begin{prop}
\label{common-component}
Let $\phi \in \cl{A}_{n}$. Then $\phi_{\CC}$ and $\phi_{\CC}^{*}$ have a common non-trivial
factor in $\cl{O}_{2n}$ if and only if $\phi$ has a non trivial factor that lies in $\cl{A}_{n\R}$.
\end{prop}
\begin{proof} On the one hand, if $f \in \cl{A}_{n\R}$ is a factor of $\phi$, then
 $f_{\CC} = f_{\CC}^{*}$ is a factor of both $\phi_{\CC}$ and $\phi_{\CC}^{*}$.
On the other hand, suppose that $G \in \cl{O}_{2n}$ is a non-unity that
divides both $\phi_{\CC}$ and $\phi_{\CC}^{*}$. Then the decomplexification $g \in \cl{A}_{n}$ of $G$
divides both $\phi$ and $\bar{\phi}$ in $\cl{A}_{n}$. That is, there are functions
$\alpha, \beta \in \cl{A}_{n}$ such that
$\phi = g \alpha$ and  $\bar{\phi} = g \beta$.
If $g$ has an irreducible   factor, say $g_{1}$, not lying
in $\cl{A}_{n\R}$, then, writing $g = g_{1} h$  for some $h \in \cl{A}_{n}$, we have
$\phi = g_{1} h \alpha$ and $\bar{\phi} =  g_{1} h \beta$.
Thus
\[ \phi = g_{1} h \alpha = \bar{g}_{1} \bar{h} \bar{\beta}. \]
This shows that $\bar{g}_{1}$ also divides $\phi$. Hence $ g _{1} \bar{g}_{1} = |g _{1}|^{2} $ is a factor
of $\phi$ in $\cl{A}_{n\R}$.
\end{proof}

\section{Pencils of integrable   $1-$forms  }
\label{section-pencil}
Let $\eta_{1}$ and $\eta_{2}$ be germs of integrable holomorphic $1-$forms at $(\C^{n},0)$, $n \geq 3$.
Suppose that they are independent --- meaning that $\eta_{1} \wedge \eta_{2} \neq 0$ --- and that
  their singular sets do not have a common component of codimension one.
 We say that $\eta_{1}$ and $\eta_{2}$ define
a \emph{pencil of integrable $1-$forms} or, shortly, an \emph{integrable pencil}  if, for every $a,b \in \C$, with at least one of them non-zero, the
$1-$form $\eta_{(a,b)} = a \eta_{1} + b\eta_{2}$ is integrable.
Thus, $\eta_{(a,b)}$ defines a holomorphic foliation of codimension one
 $\F_{t}$, where
$t = (a:b) \in \Pe^{1}$.
It may turn out that $\codim_{\C}\, \sing(\eta_{(a,b)}) = 1$, but, since $\sing(\eta_{1})$ and $\sing(\eta_{2})$ do not have a common component of codimension one, this will happen
for only finitely
  many values of $t  = (a:b)$ (see \cite[Lem. 2]{mol2011}).
For these values,  an equation for $\F_{t}$ is obtained by cancelling in $\eta_{(a,b)}$  the common factors of its coefficients. The integrability condition reads
\begin{eqnarray*} 0 = \eta_{(a,b)}\wedge d \eta_{(a,b)} & = & (a \eta_{1} + b \eta_{2}) \wedge d( a\eta_{1} + b \eta_{2}) \\
& = &
a b \left( \eta_{1} \wedge d \eta_{2} + \eta_{2} \wedge d \eta_{1} \right).
\end{eqnarray*}
Thus $\eta_{1}$ and $\eta_{2}$ define an integrable pencil   if and only if
\begin{equation}
\label{eq-pencil-condition}
\eta_{1} \wedge d \eta_{2} + \eta_{2} \wedge d \eta_{1} = 0 .
\end{equation}
This equation will be referred to   as   \emph{pencil condition}. Observe that,   knowing that $\eta_{1}$ and
$\eta_{2}$ are integrable, the pencil condition   is satisfied once
we find a pair of values $a,b \in \C^{*}$ for which   $\eta_{(a,b)}$ is integrable.   We denote the integrable pencil   engendered by $\eta_{1}$ and $\eta_{2}$
  by $\cl{P} = \cl{P}(\eta_{1},\eta_{2})$. Note that if
 $h \in \cl{O}_{n}$ is a unity, then $h \eta_{1}$ and $h \eta_{2}$  also define an integrable pencil,
 denoted by  $h \cl{P} = \cl{P}(h \eta_{1},h \eta_{2})$. Clearly, the  foliations associated to
both pencils are the same.

Given an integrable pencil $\cl{P} = \cl{P}(\eta_{1},\eta_{2})$, there exists a unique meromorphic $1-$form $\theta = \theta_{\mathcal P}$ such that, for
every $\varpi \in \cl{P}$, it holds
\begin{equation}
\label{theta}
  d \varpi = \theta \wedge \varpi.
\end{equation}
The germ of meromorphic $2-$form $d \theta$ is called  \emph{pencil curvature}.
Note that, if we multiply    $\varpi \in \cl{P}$ by a unity $h \in \cl{O}_{n}$, then
\begin{equation}
\label{eq-pencil-curvature}
  d( h \varpi ) = \left( \theta + \frac{dh}{h} \right) \wedge (h  \varpi) ,
\end{equation}
so that
  $\theta_{ h {\mathcal P}} = \theta + dh / h$ and
  the same \emph{curvature} is associated to ${\mathcal P}$ and to $h {\mathcal P}$.
The germ of
holomorphic $2-$form $\eta_{1} \wedge \eta_{2}$  is integrable and satisfies $ \eta_{1} \wedge \eta_{2} \wedge
\varpi = 0$ for every $\varpi \in \cl{P}$. This says that the codimension two holomorphic foliation $\G$
defined by $\eta_{1} \wedge \eta_{2}$ is
\emph{tangent} to all foliations associated to $1-$forms in $\cl{P}$.
The $2-$form $\eta_{1} \wedge \eta_{2}$ --- or the associated codimension two foliation --- is called    \emph{axis} of $\cl{P}$.

The next result will be the main ingredient
in the proof of Theorem \ref{teo-main-local}.   Its proof is contained, without an explicit mention,    in Cerveau's paper \cite{cerveau2002} on the so-called ``Brunella's conjecture'' for   foliations in $\Pe^{3}$.  The arguments therein   adapt  to the local framework.
Below,
 we include a sketch of the proof, which  is carried out in a more thorough way in \cite{cuzzuol2016}.
\begin{teo}
\label{pencil-foliations}
Let $\cl{P}$ be an integrable pencil  at $(\C^{n},0)$, $n \geq 3$.   Then, at
least one of the following conditions is satisfied:
\begin{enumerate}[label=(\alph*)]
\item  There exists a closed meromorphic
$1-$form $\theta$ such that $d \varpi = \theta \wedge \varpi$ for every $1-$form
$\varpi \in \cl{P}$. In particular, if $\theta$ is holomorphic, all foliations in $\cl{P}$ admit  holomorphic  first integrals.
\item   The axis of $\cl{P}$ is tangent to the levels of a non-constant meromorphic function.
\end{enumerate}
\end{teo}
\begin{proof}
Let $\theta = \theta_{\mathcal P}$ as in \eqref{theta}.
We have two cases to consider:

\par \noindent \underline{Case 1.} $d \theta = 0$ (zero curvature). This is the first of the alternatives. In the particular situation
where  $\theta$ is holomorphic, we can write $\theta = d g$ for some $g \in \cl{O}_{n}$. Putting $h = \exp(g)$,
we have that $d h/h = dg$ and thus $ h d \varpi = dh \wedge \varpi $ for every $\varpi \in \cl{P}$, giving
\[d \left( \frac{\varpi }{h} \right) = \frac{1}{h^{2}}( h d\varpi  -  d h \wedge \varpi  ) = 0.\]
Thus, after multiplication by a same unity in $\cl{O}_{n}$, all $1-$forms in the pencil become
  closed. Their integration provide the holomorphic first integrals of the statement.

\par \noindent \underline{Case 2.} $d \theta \neq 0$ (non-zero curvature). This shall give the second alternative. The arguments here are adapted from   \cite[Prop. 2]{cerveau2002}. Taking differentials in $d \varpi = \theta \wedge \varpi$, where $\varpi \in \cl{P}$, we get
$ d \theta \wedge \varpi = 0$. This is true, in particular, for $\varpi = \eta_{1}$ and $\eta_{2}$. Thus, at any point in a neighborhood of $0 \in \C^{n}$ outside
  $\sing(\eta_{1} \wedge \eta_{2})$, the $2-$forms  $d \theta$ and $\eta_{1} \wedge \eta_{2}$ are collinear.
  Thus, outside  $\sing(\eta_{1} \wedge \eta_{2})$, we can find
 a non-zero holomorphic  function $\alpha$  such that
\begin{equation}
\label{alpha}
d \theta = \alpha \eta_{1} \wedge \eta_{2}.
\end{equation}
Comparing coefficients in both sides, we conclude that $\alpha$ extends to a meromorphic function in a neighborhood of $0 \in \C^{n}$.
We have two subcases:

\par  \medskip \underline{Subcase 2.1.} $\alpha$ is constant.
Taking the exterior derivative of \eqref{alpha} and applying
the pencil condition we conclude that
$ d \eta_{1} \wedge \eta_{2} = d \eta_{2} \wedge \eta_{1} = 0$.
Thus, by \eqref{theta},
$ \theta \wedge \eta_{1} \wedge \eta_{2} = 0$.
Hence, there are meromorphic functions $\mu_{1}$ and $\mu_{2}$ such that
\[ \theta = \mu_{1} \eta_{1} + \mu_{2}\eta_{2} .\]
This applied to \eqref{theta} gives
\begin{equation}
\label{eq-d-eta}
d \eta_{1} = -\mu_{2} \eta_{1} \wedge \eta_{2} \qquad \text{and} \qquad d \eta_{2} = \mu_{1} \eta_{1} \wedge \eta_{2}.
\end{equation}
 If $\mu_{1} = 0$, then $d \eta_{2} = 0$ and there exists a non-constant $g \in \cl{O}_{n}$ such that
 $ d g  = \eta_{2}$. In particular, this $g$ is a holomorphic first integral for the axis of $\cl{P}$.
Let us suppose $\mu_{1} \neq 0$.  Write, from \eqref{eq-d-eta},
 $ d \eta_{1} = -  \left( \mu_{2} / \mu_{1} \right) d \eta_{2}$,
whose differentiation  gives
 \[ \mu_{1} d \left( \frac{\mu_{2}}{\mu_{1}}  \right) \wedge \eta_{1} \wedge \eta_{2} = 0.\]
Then   $\mu_{2} / \mu_{1}$ is a meromorphic first integral for the axis of $\cl{P}$, provided it is non-constant.
However, if $\mu_{2} / \mu_{1} = c$ is a constant, then $\varpi = \eta_{1} + c \eta_{2}$ is a closed $1-$form in $\cl{P}$.
Again, there exists a non-constant $g \in \cl{O}_{n}$ such that $d g  = \varpi$, giving, in particular, a holomorphic
first integral for the axis of $\cl{P}$.

\par  \medskip \underline{Subcase 2.2.} $\alpha$ is non-constant. The exterior derivative of equation $\eqref{alpha}$, along with \eqref{theta}, gives
$(d \alpha + 2 \alpha \theta) \wedge  \eta_{1} \wedge \eta_{2} = 0$.
 This implies that there exist    $k_{1}, k_{2} \in \cl{M}_{n}$, such that
\[ \frac{1}{2}  \frac{d \alpha}{\alpha} + \theta = k_{1} \eta_{1} + k_{2} \eta_{2} .\]
The same arguments of  \cite[Prop. 2]{cerveau2002} show that $k_{1}^{2}/\alpha$, $k_{2}^{2}/\alpha$ and $k_{1}/k_{1}$ are constant along the leaves of
the axis of $\cl{P}$, with at least one  non constant.
\end{proof}


\section{Holomorphic foliations with real meromorphic first integrals}
\label{section-meromorphic-first}

In this section we study germs of holomorphic foliations of codimension one at $(\C^{n},0)$
 that are tangent to the levels of  real meromorphic functions with real values. In the language of this paper,
we are considering   functions in $\cl{Q}_{n\R}$   whose levels define Levi-flat foliations.
Our objective is to give a proof for Theorem \ref{teo-tangent-to-levels}.

We start with a consideration on the dynamics of   subgroups of $\diff(\C,0)$, the group of   germs of biholomorphisms  at $(\C,0)$.
We refer to \cite{bracci2010, loray2005}  for a more extensive treatment of the subject.
Fixing an analytic coordinate $z$ at $(\C,0)$, an element $\phi \in \diff(\C,0)$
has a Taylor series expansion in the form $\phi(z) = \lambda z + \cdots$, where $\lambda \in \C^{*}$
and the dots stand for
the higher order terms. The number $\lambda$ is the \emph{multiplier} of $\phi$ and can be written
as $\lambda = e^{2 \pi i \alpha}$ for some $\alpha \in \C$, which is uniquely determined modulo
the sum of integer numbers. An element $\phi \in \diff(\C,0)$ can be:
\begin{itemize}
\item  \emph{Hyperbolic}, if $|\lambda| \neq 1$, corresponding to $\alpha \not \in \R$. In this case, $\phi$ is   analytically linearizable.
\item \emph{Elliptic}, if $|\lambda| = 1$ and $\alpha \in \R \setminus \Q$. Such an element is formally linearizable. When $\phi$ is analytically linearizable, we  call it an \emph{irrational rotation}.
\item    \emph{Parabolic},  if $|\lambda| = 1$ and $\alpha \in  \Q$. We also say that $\phi$ is \emph{resonant} and, in the analytically linearizable case, that it is a \emph{rational rotation}. In the particular case  $\lambda =1$ and $\phi \neq id$, we say that $\phi$ is \emph{tangent to the identity}.

\end{itemize}
Note that a germ of diffeomorphism $\phi \in \diff(\C,0)$
of finite order is necessarily parabolic.
If $G \subset \diff(\C,0)$ is an abelian  subgroup containing an element $\phi$ which is either hyperbolic or elliptic, than $G$ is linearizable in the same coordinate that linearize $\phi$. In particular, if $\phi$ is analytically
linearizable, then $G$ is also analytically linearizable. A standard argument shows that a finite subgroup $G \subset \diff(\C,0)$ is analytically linearizable. In this case, it must be cyclic and generated
 by one of its elements of highest order.

A germ of real analytic one-dimensional foliation
$\F$
at $(\C,0)$ with an isolated singularity at $0 \in \C$ is \emph{leafwise invariant} by $\phi \in \diff(\C,0)$ if, for every  sufficiently small $z \in \C \setminus \{0\}$,
 $z$ and $\phi(z)$ lie on the same leaf of $\F$.
We say that $\F$ is \emph{leafwise invariant} by   a subgroup $G  \subset \diff(\C,0)$ if  $\F$
is \emph{leafwise invariant} by every $\phi \in G$.
 In this case, we say that   $G$ is a subgroup of \emph{center type}   if  $\F$ is a foliation   of center type, meaning that
all orbits outside $0 \in \C$ are closed.

We have the following characterization of center type  subgroups of $\diff(\C,0)$:

\begin{lem}
\label{lem-dynamics}
Let $G \subset \diff (\C,0)$ be a center type   subgroup.
Then $G$ is an abelian group, without hyperbolic elements,    that
 \begin{itemize}
 \item   either contains an irrational rotation and, thus, is  analytically linearizable;
 \item or is formed only by parabolic elements of finite order.
 \end{itemize}
Besides, $G$ is a finite cyclic group if there is a second germ of real analytic foliation for which
$G$ is leafwise invariant.
\end{lem}

\begin{proof}
Since $G$ preserves closed curves around the origin,
each  orbit by $G$ of $z \neq 0$ cannot   accumulate  to the origin.
This gives at once the following facts:
\begin{enumerate}[label=(\roman*)]
\item $G$ contains no hyperbolic elements.
\item   \label{tang-id} $G$ contains no  diffeomorphisms tangent to the identity,
by the Flower Theorem \cite{camacho1978}.
\item Thus, the commutators of $G$, having multiplier $\lambda =1$,   cannot be tangent to the identity.
So they must be equal to the identity, implying that $G$ is abelian.
\item All elliptic elements are
analytically linearizable. Indeed, by  Dulac-Moussu's conjecture proved in \cite{perez-marco1997}, a non-linearizable elliptic
germ of diffeomorphism has an orbit that accumulates to the origin. \label{res}
\end{enumerate}
Consequently, if $G$ contains an elliptic  element it is analytically linearizable. Otherwise,
  $G$ is an abelian group   containing only parabolic elements, all of them of
finite order, since there are no elements
 tangent to the identity.

Consider now the last part of the statement. If  there existed
an elliptic element $\phi \in G$, then, in a coordinate $z$ of  $(\C,0)$ that linearizes $\phi$,    for each $r>0$ small, the orbits  of $G$ would be dense in the circles $z = r$, which are the leaves of the center type foliation leafwise invariant by $G$ --- and this should be the only foliation invariant by $G$.  On the other hand, $G$ cannot contain parabolic elements  of arbitrarily high order, since this would imply that intersecting leaves of the two foliations leafwise invariant by $G$ would have infinitely many points in common. This gives that
  $G$   is   finite cyclic, as we wished.
\end{proof}

\begin{cor} Let $G \subset \diff (\C,0)$ be a finitely generated center type   subgroup.
Then $G$ is analytically linearizable.
\end{cor}
\begin{proof} By    Lemma \ref{lem-dynamics}, it is enough to consider the case where $G$     is an abelian group whose elements are parabolic   of
finite order. However, the  fact that $G$ is finitely generated gives that it is  finite  and, thus,
 analytically linearizable.
\end{proof}

Let $\G$ be a germ of singular holomorphic  foliation at $(\C^{2},0)$.
We recall that there exists a reduction of singularities for $\G$  \cite{seidenberg1968}:
a proper holomorphic map
$\sigma: (\tilde{M},\D) \to (\C^{2},0)$, which is a composition of blow-ups, where $\D = \sigma^{-1}(0)$
is    a divisor consisting of a finite union projective lines with normal crossings.
It transforms $\G$ into a foliation
$\tilde{\G} =\sigma^{*} \G$ --- its \emph{strict transform} --- whose singularities are over $\D$ and
are all \emph{simple} or \emph{reduced}.  This means that, at each such point,
$\tilde{\G}$ is induced by a vector field whose linear part is non-nilpotent and,
 when its eigenvalues $\lambda_{1}, \lambda_{2}$ are both  non-zero, they satisfy $\lambda_{1} / \lambda_{2} \not\in \Q_{+}$. Such a singularity is said to be \emph{non-degenerate}.
  A simple singularity is a \emph{saddle-node} if one of the eigenvalues is zero.
  The reduction of singularities is not unique, but we can fix a minimal non-trivial
 reduction of singularities --- two   minimal reductions are isomorphic.
The number of blow-ups performed is the \emph{length} of the reduction process.
 Note that, if $\G$ has a
simple singularity at $0 \in \C^{2}$, we   have to blow-up it once in order to get a reduction of singularities and the associated length is one.
We recall the following definition of \cite{camacho1984}: $\G$ is a \emph{generalized curve foliation} if there are no saddle-nodes in its reduction of singularities.

A component  of $\D$ can be \emph{non-dicritical}, if it is
 invariant by $\tilde{\G}$, or \emph{dicritical}, otherwise. The foliation $\G$ itself is said to be
 \emph{non-dicritical} is all components of $\D$ are non-dicritical.
To each non-dicritical component
$D \subset \D$ we associate a subgroup  of   $\diff(\C,0)$ in the following way.
Take  a point $q \in D \setminus
\sing(\tilde{\G})$ and  a germ $(\Sigma,q) \simeq
(\C,0)$  of holomorphic section transversal to $D$ at $q$.
The \emph{virtual holonomy group}
is the subgroup  $\holv(\tilde{\G},D) \subset \diff(\C,0)$ formed by all germs of diffeomorphisms
$\phi \in \diff(\C,0)$ such that $z$ and $\phi(z)$ are in the same leaf of $\tilde{\G}$ for every
$z \in \Sigma$ sufficiently small.
 In principle, $\holv(\tilde{\G},D)$ depends  on the point  $q \in D \setminus
\sing(\tilde{\G})$ and on the section $\Sigma$ at $q$, but
we   get rid of this dependence by considering it  up to conjugation by
a germ of diffeomorphism in $\diff(\C,0)$.
The next two results regard the classification of a foliation by
means of its virtual holonomy. We first have:

\begin{prop}
\label{first-integral}
Let $\G$ be a germ of non-dicritical holomorphic foliation  at $(\C^{2},0)$ of generalized curve type.
Suppose that, for every component $D$ of its reduction divisor $\D$, the virtual holonomy
$\holv(\tilde{\G},D)$ is an abelian group with only finite order elements. Then $\G$ has a holomorphic first integral.
\end{prop}
\begin{proof}
The proposition is obvious if
  $\mathcal{G}$ has a simple singularity at $0 \in \C^{2}$:
 we   blow-up   once in order to
 get a reduction of singularities, use the fact  the virtual holonomy is finite and
 apply standard arguments in order to build a holomorphic first integral.
  Let us suppose the result   true for foliations with reduction of singularities of length at most $k$.
  Let $\G$ be a foliation with the properties in the statement with reduction of singularities of length $k+1$.
  We make a first blow-up
$\pi\colon (M,D)\to(\mathbb{C}^2,0)$. Let $p_1,\ldots,p_n\in D$ be the singularities of the strict transform foliation $\tilde{\G}_{0} = \pi^{*} \G$ over $D$ and denote  by  $\mathcal{G}_1$, \ldots, $\mathcal{G}_n$ the germs of $\tilde{\mathcal{G}}_{0}$ at
$p_1,\ldots,p_n$, respectively.
The induction hypothesis assures the existence of holomorphic first integrals for
 $\mathcal{G}_1$, \ldots, $\mathcal{G}_n$.
Let   $(\Sigma,q)$ be a transversal section to $D$ at $q\in D\backslash\{p_1,\ldots,p_n\}$.
Clearly, we can calculate on $\Sigma$ the virtual holonomy groups
 $H_j=  \holv(\mathcal{G}_j,D)$, $j=1\ldots,n$, which are finite groups.
 They are subgroups of $\holv(\tilde{\G}_{0},D) =\holv(\tilde{\G},D)$, the same holding for
 $H=\langle H_1,\ldots,H_n\rangle$. Now, $H$ is a finitely generated abelian group having only elements
 of finite order. Thus, it is a finite group.
 The result then follows from \cite[Lem. 3]{mattei1980}.
\end{proof}

We say that a germ of foliation $\G$ of codimension one at $(\C^{n},0)$, $n \geq 2$, is \emph{logarithmic} if there are   germs of irreducible  functions
$f_{1}, \ldots, f_{k} \in \cl{O}_{n}$ and coefficients $\lambda_{1}, \cdots, \lambda_{k} \in \C^{*}$ such
that $\G$ is defined by the meromorphic $1-$form
\begin{equation}
\label{log-1form}
 \omega = \sum_{j=1}^{k} \lambda_{j} \frac{d f_{j}}{f_{j}}.
\end{equation}
If there exists  $\mu \in \C^{*}$ such that $\mu  \lambda_{i} \in \R$ for every $i=1,\ldots,k$, we say that
$\G$ is  \emph{logarithmic with real residues}. In particular, if $\mu$ can be taken in such a way that $n_{i} = \mu  \lambda_{i} \in \Z$
for every $i=1,\ldots,k$, then
$\G$ admits the meromorphic first integral
$f = f_{1}^{n_{1}} \cdots f_{k}^{n_{k}}$, which is holomorphic if $n_{i}   \in \Z_{+}$ for every $i=1,\ldots,k$.
The following  conditions on the virtual holonomy groups are sufficient to conclude that
a holomorphic foliation is logarithmic:

\begin{prop}
\label{log-first-integral}
Let $\G$ be a germ of non-dicritical  holomorphic foliation  at $(\C^{2},0)$
of generalized curve type. Suppose that:
\begin{enumerate}[label=(\roman*)]
\item  for every component $D \subset \D$,  the virtual holonomy
$\holv(\tilde{\G},D)$ is   analytically linearizable;
\item  for some   $D \subset \D$, there exists    an irrational rotation in  $\holv(\tilde{\G},D)$.
\end{enumerate}
 Then $\G$ is a logarithmic foliation.
 Besides, if there are no hyperbolic elements in some $\holv(\tilde{\G},D)$, then $\G$ is logarithmic
 with real residues.
\end{prop}
\begin{proof} This result has been proved     in \cite{camacho1992} assuming the
existence of a hyperbolic element in place of an irrational rotation.
The proof of our version
follows the very same steps, which we describe next. Essentially, the following fact is used: if $G \subset \diff (\C,0)$ is an abelian subgroup
containing an irrational rotation $\phi$,  then the whole $G$ is   linear
in the same coordinate that linearizes $\phi$.
 We delineate a sketch of the proof in order to check how the  argument works.
First of all,   each  non-degenerate simple singularity $p \in \sing(\tilde{\G})$  is linearizable  since the holonomy of each separatrix contained in $\D$ is linearizable \cite{mattei1980}.
We have the following steps:
\begin{itemize}[leftmargin=*]
\item as in \cite[Lem. 3]{camacho1992}, the existence of an irrational rotation in  $\holv(\tilde{\G},D)$
for some $D \subset \D$ implies the existence of an irrational rotation   in  $\holv(\tilde{\G},D)$ for every in $D \subset \D$.
\item Following \cite[Prop. 1]{camacho1992}, for each $D \subset \D$, there exists a logarithmic $1-$form $\omega_{D}$ defining $\tilde{\G}$ in a neighborhood of $D$ satisfying:  \\ \smallskip
\underline{Property $(*)$}: if $(\Sigma,q) \simeq (\C,0)$ is
    a transversal section of $D$ at $q \in D \setminus \sing(\tilde{\G})$ and $z$ is an analytic coordinate at $(\C,0)$  that linearizes $\holv(\tilde{\G},D)$ at $(\Sigma,q)$, then
    $$ \omega_{D}|_{\Sigma} = dz / z .$$
    The construction of $\omega_{D}$ near the regular points of $\sing(\tilde{\G})$ on $D$ is based on the fact that the existence of an irrational rotation in $\holv(\tilde{\G},D)$ implies that,  given
    $(\Sigma_{1},q_{1})$ and  $(\Sigma_{2},q_{2})$ transversal sections as above, with coordinates
    $z_{1}$ and $z_{2}$ which linearize  $\holv(\tilde{\G},D)$, then the germ of diffeomorphism  $f_{\gamma}: (\Sigma_{1},q_{1}) \to (\Sigma_{2},q_{2})$, obtained by the lifting of a path $\gamma$ in  $D \setminus \sing(\tilde{\G})$ linking $q_{1}$ to  $q_{2}$, is a linear map. The construction of
    $\omega_{D}$ in a neighborhood of a point $p \in \sing(\tilde{\G})$ relies on the fact that
    $\tilde{\G}$ is linearizable at $p$.
\item As in \cite[Lem. 4]{camacho1992}, given two
components $D_{1}, D_{2} \subset \D$ intersecting at a point $p$,
 then there exists    $c \in \C^{*}$ such that
$\omega_{D_{1}} = c \, \omega_{D_{2}}$.   This essentially follows from   Property $(*)$
and from the fact that
$\tilde{\G}$ is linearizable at $p$.
We take coordinates $(z_{1},z_{2})$ at $p$ in which $\tilde{\G}$ is induced by the $1-$form
 $\omega =  dz_{1}/z_{1} - \lambda dz_{2}/z_{2}$, where
 $\lambda \in \C^{*} \setminus \Q_{+}$, such that $D_{1} = \{z_{1}=0\}$,  $D_{2} = \{z_{2}=0\}$.
  For small $a_{1}, a_{2}  \in \C^{*}$, by taking sections $\Sigma_{1} = \{z_{2} = a_{2}\}$,   transversal to $D_{1}$,  and $\Sigma_{2} = \{z_{1} = a_{1}\}$, transversal to $D_{2}$,   then, if $\holv(\tilde{\G},D_{1})$ is linear in the coordinate $z_{1}$,   an irrational rotation in $\holv(\tilde{\G},D_{1})$ is transferred to $\holv(\tilde{\G},D_{2})$ as a linear map in the coordinate $z_{2}$ of $\Sigma_{2}$. Thus, $\holv(\tilde{\G},D_{2})$ is linear in the coordinate $z_{2}$. This implies the result.
 \item A final   adjustment  of the constants  $c \in \C^{*}$ along the components of the desingularization divisor $\D$ gives a logarithmic $1-$form $\tilde{\omega}$  that defines
     $\tilde{\G}$ in a neighborhood of $\D$. This corresponds to a logarithmic $1-$ form $\omega$, as in \eqref{log-1form},
     that defines $\G$ in a neighborhood of $0 \in \C^{2}$.
\end{itemize}
For the last part of the statement, we   first remark that,
invoking again \cite[Lem. 3]{camacho1992},  the groups  $\holv(\tilde{\G},D)$ are   devoid of hyperbolic elements
for all components $D \subset \D$.
By construction, the meromorphic
  $1-$form $\tilde{\omega} = \sigma^{*} \omega$,
  that induces  $\tilde{\G}= \sigma^{*}\G$,
    has  simple poles over the normal crossings divisor  $(\tilde{\omega})_{\infty} = \D \cup \tilde{S}_{1} \cup \cdots \cup \tilde{S}_{k}$,
where $\tilde{S}_{i} = \sigma^{*}S_{i}$ is the strict transform of
 $S_{i} = \{f_{i} = 0\}$.
Evidently, the residue  of $\tilde{\omega}$ along each curve $\tilde{S}_{i}$ is $\lambda_{i}$.
For a simple singularity $p \in \D$ of $\tilde{\G}$,  which is non-degenerate since $\G$ is
of generalized curve type, the  ratio of eigenvalues is the negative of
  the ratio   of residues associated to the two components of
$(\tilde{\omega})_{\infty}$ intersecting at $p$.
This ratio must be real, since otherwise
the holonomy of the two separatrices of
 $\tilde{\G}$ at $p$ would   engender hyperbolic elements in the  virtual holonomy.
This fact, considered along $(\tilde{\omega})_{\infty}$, leads to the conclusion  that $\omega$ has real residues.
 \end{proof}

We now have the elements to propose a proof for Theorem \ref{teo-tangent-to-levels}.

\medskip
\smallskip

\noindent\emph{Proof of Theorem \ref{teo-tangent-to-levels}.} \
Let $\G$ be a germ of holomorphic foliation as in the assertion, tangent to
the levels of
 $f \in \cl{Q}_{n \R}$.
By taking a smooth two-dimensional section $\rho: (\C^{2},0) \to (\C^{n},0)$ generically transversal
to $\G$   and   to the indeterminacy set of $f$ (that exists by \cite{mattei1980}), we can suppose that $n=2$ and  that $0 \in \C^{2}$ is an isolated singularity of $\G$
(see, for instance, \cite[Sec. 3.3]{cerveau2011}).  Note that
if a leaf of $\G$ accumulates to the origin,
 the same holds for the level of $f$ that contains it.
If this level has real   codimension one, it defines a $\G$-invariant real analytic Levi-flat hypersurface
at $(\C^{2},0)$. As a consequence of Cerveau-Lins Neto's Theorem \cite{cerveau2011}, $\G$ admits a meromorphic first integral, proving the result.
On the other hand, if this level has real codimension two, it must be   a complex analytic curve lying
in the singular locus of $f$. This curve is also a separatrix for $\G$.
Thus, we are reduced to the following situation:
there are finitely many leaves of $\G$ accumulating to the origin, all of them separatrices of $\G$
contained in the singular locus of $f$. This means, in particular, that $\G$ is a non-dicritical
generalized curve foliation.
Take a reduction of singularities $\sigma: (\tilde{M},\D) \to (\C^{2},0)$  for $\G$ and set
$\tilde{\G} = \sigma^{*} \G$. For every component $D \subset \D$,
the virtual holonomy groups $\HV(\tilde{\G},D)$ are subgroups of center type. Thus, by Lemma \ref{lem-dynamics}, they are all abelian, either analytically linearizable containing no hyperbolic elements or   formed by finite
order resonant elements. In the latter case,  Proposition  \ref{first-integral} gives a holomorphic
first integral for $\G$. In the former,
by Proposition \ref{log-first-integral}, we find that $\G$ is a logarithmic foliation with real residues. 
Actually, in this case, all simple singularities   of   $\tilde{\G}$   must have ratio
of eigenvalues in $\Q_{+}$ (see example bellow). This implies that  the residues of $\omega$ can be taken
in $\Z_{+}$, giving a holomorphic first integral for $\G$.
\qed

\begin{example}{\rm
Suppose that $\G$ has a simple non-degenerate linearizable singularity at $0 \in \C^{2}$, being defined by the $1-$form
\[ \omega = z_{2} dz_{1} - \lambda z_{1} dz_{2}, \ \ \text{with} \ \ \lambda \in \C^{*} \setminus \Q_{+}.\]
If $\G$ has a real meromorphic first integral $f \in \cl{Q}_{n \R}$ then $\lambda \in \Q_{-}$.
Indeed, by Cerveau-Lins Neto's Theorem, as in the above proof,
we are reduced to the case where the only leaves of $\G$ accumulating to the origin are
the separatrices $z_{1}=0$ and $z_{2}=0$. This discards the options $\lambda \in \C^{*} \setminus \R$ and
$\lambda \in \R_{+} \setminus \Q_{+}$ (nodal case). If
$\lambda \in \R_{-} \setminus \Q_{-}$,
the foliation $\G$ has the multivalued first integral $z_{1} z_{2}^{-\lambda}$, so that the real
analytic $1-$form
\[ |z_{1}z_{2}|^{2}  \, \frac{d  | z_{1} z_{2}^{-\lambda}|^{2}}{|z_{1} z_{2}^{-\lambda}|^{2}} =
\bar{z_{1}} |z_{2}|^{2} dz_{1} + z_{1} |z_{2}|^{2} d \bar{z_{1}} - \lambda |z_{1}|^{2} \bar{z_{2}} dz_{2}
- \lambda |z_{1}|^{2} z_{2} d\bar{z_{2}} \]
defines a Levi-flat foliation tangent to $\G$, which is independent of
the Levi-flat foliation defined by the levels $f$. Therefore, an application of Lemma \ref{lem-dynamics}
gives that   the holonomy group of each separatrix of
$\G$ is finite, which is a contradiction. The only remaining case is $\lambda \in  \Q_{-}$.
}\end{example}

\begin{example}{\rm We illustrate  Theorem \ref{teo-tangent-to-levels} with an example of a foliation with holomorphic
first integral tangent to the levels of a non-analytic   real  meromorphic function with real levels.
Consider coordinates $(z_{1},z_{2})$ in $\C^{2}$, where $z_{j} = x_{j} + i y_{j}$ for $j=1,2$, and
 the function $f \in \cl{Q}_{2 \R} \setminus \cl{A}_{2 \R}$
defined by
\[ f(z_{1},z_{2}) = \frac{x_{1}^{2} + y_{1}^{2}}{x_{1}^{4} + y_{1}^{4}} .\]
The only level of $f$ accumulating to $0 \in \C^{2}$ is the $z_{2}$-axis. Evidently, the non-singular vertical foliation $\G$
defined by $\omega = d z_{1}$ is tangent to the levels of $f$. In order to get a truly singular holomorphic foliation with holomorphic first integral,
it suffices to take the strict transform of $\G$ by the quadratic map $\pi(z,z_{2}) = (zz_{2},z_{2})$, where
$z= x+iy$,
getting a foliation $\tilde{\G}$ defined by $\tilde{\omega}=z_{2} dz + z dz_{2}$ which is tangent to the levels of
\[ \tilde{f}(z,z_{2}) = \frac{(xx_{2}-yy_{2})^{2} + (yx_{2}+xy_{2})^{2}}{(xx_{2}-yy_{2})^{4} + (yx_{2}+xy_{2})^{4}}.\]
}\end{example}


\section{Holomorphic foliations tangent to Levi-flat foliations}
\label{section-tangent-to-Levi-flat}

Let $\F$ be a germ of real analytic Levi-flat foliation at $(\C^{n},0)$,
 defined by a germ of
real analytic $1-$form $\omega$.
Keeping the notation of Section \ref{section-preliminaries}, we denote by
$\LL = \LL(\F)$ the Levi foliation, which we suppose to be holomorphic.
In this situation,  Theorem \ref{teo-main-local} asserts that there are
two possible models for $\F$,  described in the following examples.

\begin{example}
\label{ex-closed-form}
{\rm
Let $\tau$   be a   closed  meromorphic $1-$form at $(\C^{n},0)$. We include here the case
$\tau = d \rho$, where $\rho$ is either in  $\cl{O}_{n}$ or in $\cl{M}_{n}$.
Let    $\psi \in \cl{O}_{n}$ be an equation for the  poles of $\tau$.
 If $h \in \cl{A}_{n \R}$ is a unity,
then
\[ \omega  =  h \psi \bar{\psi}    \frac{  \tau   + \bar{\tau}}{2}   = h |\psi|^{2} \re( \tau) \]
is a real analytic $1-$form that satisfies the integrability condition. It defines a
Levi-flat foliation whose Levi foliation $\LL$ is holomorphic  defined by $\tau$.
}\end{example}

\begin{example}
\label{ex-meromorphic-integral}
{\rm
Let $\rho$ be a   non-constant meromorphic (or holomorphic) function at $(\C^{n},0)$ whose levels define a holomorphic foliation $\LL$. Let $\kappa \in \cl{Q}_{n}$
be a real  meromorphic function, with $\kappa/ \bar{\kappa}$  constant along the leaves of $\LL$, such that
the $1-$form $\eta = \kappa d \rho$  is real analytic
without non-trivial factors in $\cl{A}_{n \R}$.
Define the real analytic $1-$form
\[ \omega =  \frac{ 1}{2} \left(  \kappa d \rho  + \bar{\kappa} d \bar{\rho} \right)
=  \re(\kappa  d \rho). \]
Observe that the integrability condition,
\[  \omega \wedge d \omega = \frac{1}{4} \left( \bar{\kappa} d \kappa - \kappa d \bar{\kappa} \right) \wedge d \rho
\wedge d \bar{\rho} = 0,
\]
is equivalent to
$\kappa/\bar{\kappa}$ being constant along the leaves of $\LL$. Thus, $\omega$ defines a Levi-flat foliation
whose Levi foliation is $\LL$.
Note that $\LL$ admits real meromorphic first integrals
in  $\cl{Q}_{n \R}$ (for instance, $\re(\rho)$ and
$\im(\rho)$, as well as $\re(\kappa/ \bar{\kappa})$ and $\im(\kappa/ \bar{\kappa})$, if these are non-constant).
}\end{example}

As seen in Subsect. \ref{section-levi}, the Levi foliation $\LL$
is tangent to the distribution of complex hyperplanes
 defined by the   real analytic  $1-$form $\eta$ of type $(1,0)$ such that
$\omega = \re(\eta)$ (equation \eqref{eq-omega}). By Lemma \ref{lem-holomorphic-L}, the
 condition of $\LL$ being holomorphic is equivalent to existence
of $\phi \in \A_{n}$ such that
$ \eta = \phi \sigma$,
where  $\sigma$  is an integrable holomorphic $1-$form at $(\C^{n},0)$ that defines $\LL$.
Writing
$\sigma =  \sum_{j=1}^{n} \alpha_{j}(z) dz_{j}$,
where $\alpha_{j} \in \cl{O}_{n}$, for $j=1,\ldots,n$, are without common factors, we get,
by taking mirrors, the integrable holomorphic $1-$form $\sigma^{*} =  \sum_{j=1}^{n} \alpha_{j}^{*}(w) dw_{j}$
that defines the foliation $\LL^{*}$ in $\C^{n*}$.
Their complexifications produce two product foliations  at  $(\C^{n} \times \C^{n*},0) \simeq (\C^{2n},0)$:
\begin{itemize}
\item $\LL \times \C^{n*}$, defined by
  $\sigma_{\CC}$,    whose leaves are vertical cylinders over the leaves of $\LL$;
\item $\C^{n} \times \LL^{*}$, defined by   $\sigma^{*}_{\CC}$,    whose leaves are horizontal cylinders over the leaves of the
mirror foliation $\LL^{*}$.
\end{itemize}
Now, the  complexification  of $\omega$ is  a germ of integrable holomorphic $1-$form $\omega_{\CC}$
at $(\C^{2n},0)$ which defines the holomorphic foliation
of codimension one $\F_{\CC}$,  the complexification of $\F$. We have
\begin{equation}
\label{complexification-pencil}
 \omega_{\CC} = \frac{1}{2}(\eta + \bar{\eta})_{\CC} =
\frac{1}{2}( \eta_{\CC} + \eta_{\CC}^{*})  = \frac{1}{2}\phi_{\CC} \sigma_{\CC} + \frac{1}{2}\phi_{\CC}^{*}\sigma^{*}_{\CC}.
 \end{equation}
 In particular, $\omega_{\CC}$ is $(*)$-symmetric, that is $\omega_{\CC}^{*} = \omega_{\CC}$.
On the other hand, the complexification $\LL_{\CC}$ of $\LL$ is the holomorphic foliation of codimension
two at  $(\C^{2n},0)$ defined by
$  (\eta \wedge \bar{\eta})_{\CC} =
\eta_{\CC} \wedge \eta_{\CC}^{*}$ --- that is to say, $\LL_{\CC}$ is the product foliation $\LL \times \LL^{*}$. Since $\LL$ is tangent to $\F$, we have that $\LL_{\CC}$ is tangent to $\F_{\CC}$ --- and also to $\LL \times \C^{n*}$ and to $\C^{n} \times \LL^{*}$.

Observe that  both
$\eta_{\CC}$ and $\eta_{\CC}^{*}$ may have   codimension
one components in their singular sets, given by $\phi_{\CC} = 0$ and $\phi_{\CC}^{*}=0$. However,
$\phi_{\CC}$ and  $\phi_{\CC}^{*}$ are relatively prime in $\cl{O}_{2n}$, since otherwise,
by Proposition \ref{common-component}, the $1-$form $\omega$ would have a non-trivial
factor in $\cl{A}_{n \R}$. Hence,
  equation
\eqref{complexification-pencil} says that
  $\eta_{\CC} = \phi_{\CC} \sigma_{\CC}$ and $\eta_{\CC}^{*} = \phi_{\CC}^{*}\sigma^{*}_{\CC}$ define an integrable pencil $\cl{P} = \cl{P}(\eta_{\CC},\eta_{\CC}^{*})$
 that contains $\omega_{\CC}$, whose axis is the codimension two foliation   $\LL_{\CC} = \LL \times \LL^{*}$.
This geometric  fact  is the core of the proof of Theorem \ref{teo-main-local}, which we present next.

\medskip
\smallskip

\noindent\emph{Proof of Theorem \ref{teo-main-local}.} \
We  apply Theorem \ref{pencil-foliations} to the   integrable   pencil  $\cl{P} = \cl{P}(\eta_{\CC},\eta_{\CC}^{*})$.   The alternatives therein are in direct correspondence
with those in the theorem's statement.
Let us examine them.
\smallskip
\par \noindent \underline{Alternative \emph{(a)}}.
There exists a closed meromorphic $1-$form $\theta$ such that $d \varpi = \theta \wedge \varpi$
for every $1-$form $\varpi$ in $\cl{P}$.
In particular, we have the equations
\begin{equation}
\label{liouville} d \omega_{\CC} = \theta \wedge \omega_{\CC}  \qquad \text{and}  \qquad
 d \eta_{\CC}  = \theta \wedge \eta_{\CC}   \qquad \text{and} \qquad
  d \eta_{\CC}^{*}  = \theta \wedge \eta_{\CC}^{*} \, ,
 \end{equation}
which, by mirroring, turn into
\[  d \omega_{\CC} = \theta^{*} \wedge \omega_{\CC} \qquad \text{and} \qquad
 d \eta_{\CC}^{*}  = \theta^{*} \wedge \eta_{\CC}^{*}   \qquad \text{and} \qquad
  d \eta_{\CC}   = \theta^{*} \wedge \eta_{\CC} \, .\]
Taking into account these sets of equations,
  the closed $1-$form $\vartheta = \theta - \theta^{*}$
satisfies
\[  \vartheta \wedge \omega_{\CC}  =
 \vartheta \wedge \eta_{\CC}   =  \vartheta \wedge \eta_{\CC}^{*} = 0 .\]
Since $\omega_{\CC}$, $\eta_{\CC}$ and $\eta_{\CC}^{*}$ are independent, this is possible if and only if $\vartheta = 0$. Hence
$\theta$ is $(*)$-symmetric: $\theta = \theta^{*}$.

If $\theta$ is holomorphic then  
$\theta = d G$, for some  $G \in \cl{O}_{2n}$. Since $\theta = \theta^{*}$,
we also have that $\theta = d(G + G^{*})/2$. Thus,
replacing $G$ by
$(G + G^{*})/2$, we can suppose that $G$ is
$(*)$-symmetric. Then the $(*)$-symmetric function $H = \exp(G)$ is a unity in $\cl{O}_{2n}$
such that
 $ d \left( \varpi / H \right) = 0$ for every
$\varpi \in \cl{P}$. In particular $d \left( \eta_{\CC} / H \right) = 0$, so
we can find $F \in  \cl{O}_{2n}$ such that $\eta_{\CC}/ H = d F$.
Since $\eta_{\CC}$ is a $1-$form of type $(1,0)$, the same is true for $d F$. This means that,
in the variables $(z,w) \in \C^{2n}$,
$F$ is a function only in   $z$.
We have
\[ \frac{\omega_{\CC}}{H} =
\frac{1}{H} \left( \frac{\eta_{\CC} + \eta_{\CC}^{*}}{2} \right) = \frac{dF + d F^{*}}{2}.\]
Decomplexifying this expression, we get
\[ \frac{\omega}{h} =  \frac{df + d \bar{f}}{2} = \re( df),\]
where $h \in \cl{A}_{n\R}$ and $f \in \cl{O}_{n}$ are such that
$h_{\CC} = H$ and $f_{\CC} = F$.
This a particular case of item (a) in the statement,
with $\tau = df$ and $\psi = 1$
(or of item (b) with $\rho = f$ and $\kappa = h$).

Now, when the   $1-$form $\theta$ is purely meromorphic,
by \cite{cerveau1982} it
 can be written
as
\begin{equation}
\label{eq-closed-form}
\theta = \sum_{j=1}^{\ell} \lambda_{j} \frac{d F_{j}}{F_{j}} + d \left( \frac{G}{F_{1}^{k_{1}} \cdots F_{k}^{k_{\ell}}} \right) ,
\end{equation}
where $F_{1}, \ldots, F_{\ell} \in \cl{O}_{2n}$ are irreducible equations for the components of the polar set, $G  \in \cl{O}_{2n}$,
$\lambda_{1}, \ldots, \lambda_{\ell} \in \C$ and $k_{1}, \ldots, k_{\ell} \in \Z_{\geq 0}$.
Using that $\theta$ is $(*)$-symmetric,   we can refine this writing in the following way:
\begin{equation}
\label{closed-form}
  \theta = \sum_{j=1}^{r} \mu_{j} \frac{d h_{j}}{h_{j}} + \sum_{j=1}^{s} \left(\lambda_{j} \frac{d f_{j}}{f_{j}} +
\bar{\lambda}_{j} \frac{d f_{j}^{*}}{f_{j}^{*}} \right)+ d \left( \frac{G}{h_{1}^{m_{1}}
\cdots h_{r}^{m_{r}} \left(f_{1} f_{1}^{*} \right)^{n_{1}}
 \cdots  \left(f_{s} f_{s}^{*} \right)^{n_{s}} } \right),
 \end{equation}
where $h_{1}, \ldots, h_{r}, G   \in \cl{O}_{2n}$ are $(*)$-symmetric,
$f_{1}, \ldots, f_{s} \in \cl{O}_{2n}$ are not $(*)$-symmetric,
$\mu_{1}, \ldots, \mu_{r} \in \R$,
$\lambda_{1}, \ldots, \lambda_{s} \in \C$ and
$m_{1}, \ldots, m_{r}, n_{1}, \ldots, n_{s} \in \Z_{\geq 0}$.

Using  $ d \eta_{\CC}  = \theta \wedge \eta_{\CC}$    and $\eta_{\CC} = \phi_{\CC} \sigma_{\CC}$,
we have
\[  d   \phi_{\CC} \wedge \sigma_{\CC} +     \phi_{\CC} d \sigma_{\CC}   =
\phi_{\CC} \theta \wedge  \sigma_{\CC},  \]
which gives
\begin{equation}
\label{derivative}
     d \sigma_{\CC}   =
\left(    \theta - \frac{d   \phi_{\CC}}{\phi_{\CC}}  \right) \wedge  \sigma_{\CC} .
\end{equation}
This expression, when mirrored, becomes
\begin{equation}
\label{derivative*}
 d \sigma_{\CC}^{*}   =
\left(    \theta - \frac{d   \phi_{\CC}^{*}}{\phi_{\CC}^{*}}  \right) \wedge  \sigma_{\CC}^{*} .
\end{equation}
The following  fact is well known: if $\varpi$ is an integrable holomorphic $1-$form, with $\sing(\varpi)$
of codimension at least two, and   $\vartheta$   is a closed meromorphic $1$-form such that
 $d \varpi = \vartheta \wedge \varpi$, then the poles of $\vartheta$ are invariant by the foliation induced by $\varpi$. Applying this to \eqref{derivative} and to \eqref{derivative*}, we find that the
 poles of  $\theta -  d   \phi_{\CC} / \phi_{\CC}$ are invariant
 by the vertical foliation $\LL \times \C^{n*}$ induced by  $\sigma_{\CC}$, while
 the poles of  $\theta -  d   \phi_{\CC}^{*} / \phi_{\CC}^{*}$ are invariant by the
 horizontal foliation $ \C^{n} \times \LL^{*} $ induced by
 $\sigma_{\CC}^{*}$.

Now, regarding the expression  \eqref{closed-form} of $\theta$, we can say the following:
 \begin{itemize}
  \item there are no poles of the form $h_{j} = 0$, for $h_{j} \in \cl{O}_{2n}$ $(*)$-symmetric.
Indeed, otherwise  $h_{j}$ would be a factor of both $\phi_{\CC}$ and $\phi_{\CC}^{*}$, since $h_{j} = 0$ is   invariant by neither $\LL \times \C^{n*}$ nor $ \C^{n} \times \LL^{*}$.
The
  decomplexification of $h_{j}$  would thus engender
 a non-trivial factor of $\omega$ in $\A_{n\R}$.
 \item In other words, the poles of $\theta$ are either horizontal or vertical. Thus, in coordinates $(z,w) \in \C^{2n}$, we can suppose that, for each $j=1, \ldots, s$, there exists a unity
 $h_{j} \in \cl{O}_{2n}$  such that $f_{j} = h_{j} \tilde{f}_{j}$, where $\tilde{f}_{j} = \tilde{f}_{j}(z) \in \cl{O}_{2n}$
 is a function only in the variables $z$. As a consequence,
  $f_{j}^{*} =  h_{j}^{*} \tilde{f}_{j}^{*}$,
  where $\tilde{f}_{j}^{*} = \tilde{f}_{j}^{*}(w) \in \cl{O}_{2n}$
 is a function only in the variables $w$.
 \item All poles must then be cancelled by either $d   \phi_{\CC} / \phi_{\CC}$ or by
  $ d   \phi_{\CC}^{*} / \phi_{\CC}^{*}$. Thus, they are all simple and $n_{1} = \cdots = n_{s} = 0$.
 \item For this same reason, there exist a  unity $g_{0} \in \cl{O}_{2n}$ and $r_{1}, \ldots, r_{s} \in \Z^{*}$ such that
 \[ \phi_{\CC} = g_{0}  ( \tilde{f}_{1}^{*} )^{r_{1}} \cdots  (  \tilde{f}_{s}^{*} )^{r_{s}}
 = g_{0} \, \Psi^{*}  \qquad \text{and} \qquad
 \phi_{\CC}^{*} = g_{0}^{*} \tilde{f}_{1}^{r_{1}} \cdots  \tilde{f}_{s}^{r_{s}}
 = g_{0} ^{*} \, \Psi, \]
 where $\Psi =  \tilde{f}_{1}^{r_{1}} \cdots \tilde{f}_{s}^{r_{s}} \in \cl{O}_{2n}$ is a function
 in the variables $z$ only. Besides,
  $\lambda_{j} = r_{j} \in \Z^{*}$   for all $j=1,\ldots,s$.
  \item  The exact part can be written in the form $dG = d h_{0} / h_{0}$, where $h_{0} = \exp(G)$ is a $(*)$-symmetric unity in $\cl{O}_{2n}$.
 \end{itemize}
Taking into account all these comments, we can write
\[ \theta =  \sum_{j=1}^{s} r_{j} \left( \frac{d(h_{j} \tilde{f}_{j})}{h_{j} \tilde{f}_{j}}
+  \frac{d(h_{j}^{*} \tilde{f}_{j}^{*} )}{h_{j}^{*}  \tilde{f}_{j}^{*} } \right) + \frac{d h_{0}}{h_{0}}
=  \frac{d (\Psi \Psi^{*})}{\Psi \Psi^{*}} + \frac{d H}{H} , \]
where
 $H = h_{0} ( h_{1} h_{1}^{*})^{r_{1}} \cdots (h_{s}h_{s}^{*})^{r_{s}}$
 is a $(*)$-symmetric unity in $\cl{O}_{2n}$.
Thus, by considering the integrable pencil $(1/H) \cl{P}$ instead of $\cl{P}$ ---
 which corresponds to performing the above calculations to the Levi-flat $1-$form $(1/h) \omega$, where $h \in \A_{n \R}$ is
 a unity such that $h_{\CC} = H$ --- we can suppose that
 \[ \theta =   \theta_{\mathcal P} =   \frac{d (\Psi \Psi^{*})}{\Psi \Psi^{*}}.\]

Inserting this   in   equations  \eqref{liouville}, we find
\begin{equation}
\label{eq-3closed-forms}
 d \left( \frac{1}{ \Psi \Psi^{*}} \omega_{\CC} \right) =
d \left( \frac{1}{\Psi \Psi^{*}} \eta_{\CC} \right) =
d \left( \frac{1}{\Psi \Psi^{*}} \eta_{\CC}^{*} \right) = 0 .
\end{equation}
Let us define
$ \zeta_{\CC} =   \eta_{\CC}/ ( g_{0}  \Psi \Psi^{*} ) = 
 \sigma_{\CC} / \Psi$ and, symmetrically,
 $\zeta_{\CC}^{*} =  \eta_{\CC}^{*}/ (g_{0}^{*}  \Psi \Psi^{*})  =   \sigma_{\CC}^{*} / \Psi^{*}$.
Observe that, with respect to the decomposition $\C^{n} \times \C^{n*} \simeq \C^{2n}$, the $1-$form   $\zeta_{\CC}$ is of type $(1,0)$, expressed only in the variable
$z$, whereas  $\zeta_{\CC}^{*}$, of type $(0,1)$, is written  only in the variable $w$.
The second  closed differential form in \eqref{eq-3closed-forms} then reads
\[ 0 =  d \left( g_{0} \zeta_{\CC} \right)
=  d   g_{0}   \wedge \zeta_{\CC}  +   g_{0} \, d \zeta_{\CC}. \]
Decomposing $d = \partial_{z} + \partial_{w}$ into exterior derivatives with respect
to the variables $z$ and $w$ and assembling $2-$forms of the same type, we find that
\[  \partial_{w}   g_{0}   \wedge \zeta_{\CC} = 0 .\]
This gives $\partial_{w} g_{0} = 0$, implying that $g_{0} = g_{0}(z)$ depends exclusively on the variables $z$.

Let us   define $\tau_{\CC} =  g_{0} \zeta_{\CC}  =  \eta_{\CC}/ ( \Psi \Psi^{*} ) =  g_{0} \sigma_{\CC} / \Psi$. It is  a meromorphic $1-$form in the variable $z$, closed by \eqref{eq-3closed-forms}.
Returning to formula \eqref{complexification-pencil}, we  get
\[ \omega_{\CC} = \frac{\eta_{\CC} + \eta_{\CC}^{*}}{2} = \Psi \Psi^{*} \frac{\tau_{\CC} + \tau_{\CC}^{*}}{2} .\]
Decomplexifying this expression, we find
\[ \omega  =  \psi \bar{\psi}  \frac{\tau  + \bar{\tau}}{2} = |\psi|^{2} \re(\tau) ,\]
where $\psi \in \cl{O}_{n}$ is such that $\psi_{\CC} = \Psi$ and
  $\tau$ is a germ of closed meromorphic $1-$form at $(\C^{n},0)$ whose complexification is $\tau_{\CC}$.
Observe that
 $\tau$, being a meromorphic multiple of $\sigma$, induces the Levi foliation $\LL$
and also  that $\psi$ is an equation for the poles of $\tau$.
Finally, reincorporating the unity $h \in \A_{n \R}$, we get that $\omega$ has the desired shape.

\smallskip
\par \noindent \underline{Alternative \emph{(b)}}.
 The axis $\LL \times \LL^{*}$ of the pencil $\cl{P}$ has a meromorphic first integral.
This means that there exists a non-constant germ of meromorphic function $F$ at $(\C^{n} \times \C^{n*},0)$
such that
\[ d F \wedge \eta_{\CC}  \wedge \eta_{\CC}^{*} = 0 .\]
We can suppose that   $F$ is $(*)$-symmetric. If this is not so, we
apply  the $(*)$-operator,  finding
\[ d F^{*} \wedge \eta_{\CC}  \wedge \eta_{\CC}^{*} = 0. \]
Thus
\[ d \left( F + F^{*} \right) \wedge \eta_{\CC}  \wedge \eta_{\CC}^{*} = 0 .\]
If  $F + F^{*}$ is non-constant, replacing $F$ by  $F + F^{*}$, we have a $(*)$-symmetric
meromorphic first integral for $\LL \times \LL^{*}$.
On the other hand,  $F + F^{*}$ is constant if and only if
 $F + F^{*} = c$ for some $c \in \R$. Thus, putting $F - c/2$ in place of $F$,
we can suppose that $F + F^{*} = 0$. Thus, the non-constant meromorphic
function $i F$ is such that $(i F)^{*} = (-i)F^{*} = i F$.

Now, by Proposition \ref{prop-meromorhpic-function}, we can write $F = G/H$,
where $G,H \in \cl{O}_{2n}$ are relatively prime $(*)$-symmetric functions.
Let $g,h \in \cl{A}_{n\R}$ be   real analytic functions such that
$g_{\CC} = G$ and $h_{\CC} = H$. Thus, $f = g/h$ is a non-constant real meromorphic function
in $\cl{Q}_{n \R}$ such that
\[ d f \wedge \eta   \wedge \bar{\eta} = 0 ,\]
which means that $\LL$ is tangent to the levels of $f$.
By Theorem \ref{teo-tangent-to-levels}, $\LL$ has a meromorphic first integral $\rho \in \cl{M}_{n}$.
Therefore, there exists
  a real meromorphic function $\kappa \in \cl{Q}_{n}$ such that
$\eta =   \kappa d \rho .$
Consequently
 \[ \omega =  \frac{1}{2} \left( \kappa d \rho  + \bar{\kappa} d \bar{\rho}   \right)
= \re(\kappa d \rho). \]
As seen in Example \ref{ex-meromorphic-integral}, the integrability   for $\omega$  implies that $\kappa/ \bar{\kappa}$
   is constant along the leaves of $\LL$. This finishes the proof.
\qed



\section{The global case: algebraic Levi-flat foliations}
\label{section-algebraic}

Throughout this section we consider the projective space
$\Pe^{n} = \Pe^{n}_{\C}$ obtained by the usual identification of points
$\C^{n+1} \setminus \{0\}$ lying in the same line through the origin. We identify geometric
objects in $\Pe^{n}$ with their cones in $\C^{n+1} \setminus \{0\}$ and work
  in coordinates $z = (z_0,\ldots,z_{n}) \in \C^{n+1}$.

In this way, a real algebraic foliation $\F$ of codimension one in $\Pe^{n}$  is
defined in $\C^{n+1}$ by an integrable real $1-$form $\omega$ with relatively prime polynomial coefficients satisfying the following conditions:
\begin{enumerate}[label={(\roman*)}]
\item   $i_{\bm{r}}\omega  = i_{\bm{r}}\omega^{\sharp}=0$, where $i_{\bm{r}}$ denotes the contraction by the real radial vetor
field; \smallskip
\item $\left(\Lambda_{\lambda}\right)^{*} \omega = |\lambda|^{2d} \omega$ for some $d \geq 1$ and for every $\lambda \in \C^{*}$, where $\Lambda_{\lambda}$ is the homothety of $\C^{n}$ given by   $\Lambda_{\lambda}(z) = \lambda z$.
\end{enumerate}

Employing the notation of Subsection \ref{section-levi}, we write
\[\omega =    \sum_{j=0}^{n} \left( A_{j} d x_{j} + B_{j} d y_{j} \right) \qquad \text{and} \qquad
\omega^{\sharp} =    \sum_{j=0}^{n} ( - B_{j} d x_{j} + A_{j} d y_{j} )
,\]
where $A_{j}, B_{j} \in \C[z,\bar{z}]$ are such that $A_{j} = \bar{A}_{j}$
and $B_{j} = \bar{B}_{j}$ for every $j=0,\ldots,n$.
We also take the canonical decomposition $\omega= (\eta + \bar{\eta})/2$,
where
\[\eta  =  \sum_{j=0}^{n} \frac{A_{j} - i B_{j}}{2} dz_{j}
=  \sum_{j=0}^{n} \varepsilon_{j} dz_{j}.\]
Condition  (i)  expresses  that   $\F$ contains all complex lines through the origin of $\C^{n+1}$.
In  coordinates, $\bm{r} = \sum_{j=0}^{n} \left( x_{j} \partial/ \partial x_{j} +  y_{j} \partial/ \partial y_{j} \right)$ and we have
 \[  \sum_{j=0}^{n} (x_{j} A_{j}   +  y_{j}B_{j})    = 0 \qquad \text{and} \qquad
 \sum_{j=0}^{n}( - x_{j} B_{j}   +  y_{j} A_{j})   = 0 .\]
 This is the same of asking the vanishing of the contraction of $\eta$ by the
complex radial vetor
field $\bm{R} = \sum_{j=0}^{n}   z_{j} \partial/ \partial z_{j}$, that is,
$\sum_{j=0}^{n} z_{j} \varepsilon_{j}    = 0 $.

On its turn, condition (ii) asserts that the distribution of real hyperplanes induced by $\omega$
descends, in a  well defined way, to $\Pe^{n}$.
Suppose, initially, that the coefficients of $\eta$,  $\varepsilon_{j} \in \C[z,\bar{z}]$, are bihomogeneous of bidegree $(d-1,d)$, for some $d \geq 1$. Then, for $\lambda \in \C^{*}$,  we have
\begin{equation}
\label{eq-lambda-eta}
 \left(\Lambda_{\lambda}\right)^{*} \eta = \lambda \sum_{j=1}^{n} \lambda^{d-1} \bar{\lambda}^{d}\varepsilon_{j}(z,\bar{z}) dz_{j}
= |\lambda|^{2d} \eta .
\end{equation}
Similarly,
$ \left(\Lambda_{\lambda}\right)^{*} \bar{\eta} =  |\lambda|^{2 d } \bar{\eta},$
which gives
\[ \left(\Lambda_{\lambda}\right)^{*} \omega = \left(\Lambda_{\lambda}\right)^{*} \left(\frac{  \eta + \bar{\eta}}{2} \right) =   |\lambda|^{2d} \omega.\]
Reciprocally, if condition (ii) holds,  it is straightforward that
 the coefficients $\varepsilon_{j}$ of $\eta$ must be bihomogeneous of bidegree $(d-1,d)$.

Let us suppose that $\F$ is an algebraic foliation
  in $\Pe^{n}$ that is Levi-flat, meaning that   it is a  local Levi-flat  foliation
 at each point of $\Pe^{n}$. Thus, the Levi foliation $ \LL  = \LL(\F)$ is a real analytic complex foliation in $\Pe^{n}$  of complex codimension one. In consonance with the local study of
 the previous sections, we   suppose that $\LL$ is a holomorphic foliation. Thus, we associate to $\LL$
  a degree $d_{0} \geq 0$, which counts its tangencies, with multiplicities, with a generic line
in $\Pe^{n}$.
In   coordinates $z=(z_0,\ldots,z_{n}) \in \C^{n+1}$, the foliation $\LL$ is induced by an integrable polinomial $1$-form
$\sigma = \sum_{j=0}^{n} \alpha_{j}(z) dz_{j}$, whose coefficients  $\alpha_{j} \in \C[z]$ are relatively prime homogeneous polynomials of degree $d_{0} +1$, which contracts to zero by the complex radial vector field, that is,
$\sum_{j=0}^{n} z_{j} \alpha_{j}    = 0 $.
Since $\eta$ defines the same distribution of complex hyperplanes as $\sigma$,   there exists a bihomogeneous polynomial $\varphi \in \C[z,\bar{z}]$ of bidegree $(d  - d_{0} -2, d ) $
such that
\[ \eta = \varphi \sigma.\]
In particular, $d_{0} \leq d + 2$.

Now we analyze the  two alternatives of Theorem \ref{teo-main-local} for
$\F$, considered as a local foliation  at  $0 \in \C^{n+1}$:

\smallskip
\par \noindent \underline{Alternative \emph{(a)}}.
$ \omega  =  h |\psi|^{2} \re( \tau)$, where
$\tau$ is a closed meromorphic $1-$form that induces $\G$,
  $\psi \in \cl{O}_{n+1}$
is an equation for   polar set of $\tau$ and
  $h \in \cl{A}_{n+1,\R}$ is a unity.
  The $1-$form $\tau$ goes down to $\Pe^{n}$, that is, for every $\lambda \in \C^{*}$  the $1-$forms $\tau$ and $\left(\Lambda_{\lambda}\right)^{*}   \tau$ define the same
distribution of complex hyperplanes  and $i_{\bm{R}} \tau = 0$, where $\bm{R}$ is the complex radial vector field.
Hence, $\tau$ has a writing as
in \eqref{eq-closed-form} (we borrow the notation, putting $\tau$ instead of $\theta$), where $F_{1},\ldots,F_{\ell},G \in \C[z]$ are homogeneous polynomials
satisfying:
\begin{itemize}
\item $ \deg G = k_{1} \deg F_{1}+ \cdots + k_{\ell} \deg F_{\ell} $; \smallskip
\item $\lambda_{1} \deg F_{1}+ \cdots + \lambda_{\ell} \deg F_{\ell} = 0$ (Residue Theorem). \smallskip
\end{itemize}
The function $\psi$ is an equation of the poles of $\tau$, thus it is a homogeneous
polynomial in $\C[z]$. Finally, the homogeneity of
$ \omega$, $\psi$ and $\tau$ gives that
$h$ must be a  constant, supposed to be 1.

\smallskip
\par \noindent \underline{Alternative \emph{(b)}}. $ \omega  =  \re(\kappa d \rho)$, where
   $\rho \in \cl{M}_{n+1}$ is a complex meromorphic function,
$\kappa \in \cl{Q}_{n+1}$ is
a real meromorphic function such that $\kappa/\bar{\kappa}$ is constant along the leaves of $\G$. Again, since $\rho$ is constant along the leaves of a foliation that goes down to $\Pe^{n}$, all leaves of $\LL$ are algebraic and $\rho$ must be a complex rational function in $\Pe^{n}$, given in homogeneous
coordinates as the quotient of relatively prime homogeneous polynomials in $\C[z]$ of the same degree. From \eqref{eq-lambda-eta}, considering that
$\eta = \kappa d \rho$, we see that $\kappa$ must be a real rational function of bidegree $(d,d)$
defined as the quotient of relatively prime homogeneous polynomials in $\C[z,\bar{z}]$.


We summarize this discussion in the following    global version of Theorem \ref{teo-main-local}:
\begin{maintheorem}
\label{teo-main-global}
Let $\F$ be a  real algebraic Levi-flat foliation on  $\Pe^{n}$, $n \geq 2$, induced in
homogeneous coordinates  $z= (z_{0},\cdots,z_{n}) \in \C^{n+1}$ by an integrable   $1-$form $\omega$
with polynomial coefficients satisfying  conditions (i) and (ii) above.
If the Levi foliation $\LL = \LL(\F)$ is holomorphic, then   at least one of the
following two possibilities occurs:
\begin{enumerate}[label=(\alph*)]
\item $\omega  =  |\psi|^{2} \re( \tau)$, where $\tau$ is closed complex rational $1-$form that defines $\LL$ and $\psi \in \C[z]$ is a homogeneous
equation for   polar set of $\tau$.
\item $ \omega  =   \re(\kappa d \rho)$, where $\rho$ is a complex rational function   that is a first integral for $\LL$,
$\kappa$ is a real rational function of bidegree $(d,d)$, for some $d \geq 1$, defined as a quotient of  relatively prime homogeneous polynomials in $\C[z,\bar{z}]$, such that $\kappa/ \bar{\kappa}$ is constant
along the leaves of $\G$.
\end{enumerate}
\end{maintheorem}

We finish with an  illustration on how  to produce examples of the situation $(b)$ in the theorem:
\begin{example} {\rm
 Let $\rho = F/G$ be a complex rational function, where  $F,G \in \C[z]$ are relatively prime homogeneous polynomials of the same degree, whose levels define a holomorphic foliation $\G$ on $\Pe^{n}$. Let $\upsilon = R/S$ be a complex rational function in the variables
$u=(u_{1},u_{2}) \in \C^{2}$, defined as a quotient of relatively prime homogeneous polynomials $R,S \in \C[u]$.
Then $\upsilon \circ \rho = R(F,G)/S(F,G)$ is also a rational first integral for $\G$.
Suppose also that $u_{2}^{2}$ factors $R$, so that $G^{2}$ divides $R(F,G)$, making
 $R(F,G) d \rho$   a polynomial $1-$form.
Let $\kappa = R(F,G) \xbar{S(F,G)}$. It is a homogeneous polynomial in $\C[z,\bar{z}]$ of bidegree $(d,d)$, where $d = \deg(R(F,G)) = \deg(S(F,G))$, such that $\kappa/ \bar{\kappa} = \upsilon \circ \rho /
\xbar{\upsilon \circ \rho}$ is constant along the leaves
of $\G$. Thus,   up to cancelling    factors of the form $|\varphi|^{2}$, where $\varphi \in \C[z]$ is a homogeneous factor of $F$ or $G$,
 the  $1-$form $\omega  =   \re(\kappa d \rho)$  defines a Levi-flat foliation on $\Pe^{n}$ whose
underlying foliation is $\G$.
}\end{example}

\bibliographystyle{plain}
\bibliography{referencias}

\medskip  \medskip  \medskip

\noindent
Arturo Fern\'andez-P\'erez\\
Departamento de Matem\'atica \\
Universidade Federal de Minas Gerais \\
Av. Ant\^onio Carlos, 6627  \  C.P. 702  \\
30123-970  --
Belo Horizonte -- MG,
BRAZIL \\
fernandez@ufmg.br

\medskip \medskip \medskip

\noindent
Rog\'erio  Mol  \\
Departamento de Matem\'atica \\
Universidade Federal de Minas Gerais \\
Av. Ant\^onio Carlos, 6627  \  C.P. 702  \\
30123-970  --
Belo Horizonte -- MG,
BRAZIL \\
rmol@ufmg.br

\medskip \medskip  \medskip
\noindent
Rudy Rosas \\
Pontificia Universidad Cat\'olica del Per\'u \\
Av. Universitaria 1801  \\
Lima,
Peru \\
rudy.rosas@pucp.edu.pe

\end{document}